\documentclass{amsart}
\usepackage{amssymb,amsfonts,latexsym}

\newtheorem{theorem}{Theorem}[section]
\newtheorem{lemma}[theorem]{Lemma}
\newtheorem{proposition}[theorem]{Proposition}
\newtheorem{corollary}[theorem]{Corollary}
\newtheorem{claim}[theorem]{Claim}
\theoremstyle{definition}
\newtheorem{definition}[theorem]{Definition}

\newtheorem{example}[theorem]{Example}

\newtheorem{remark}[theorem]{Remark}
\newtheorem{general remarks}[theorem]{General remarks}
\newtheorem{note}[theorem]{Note}

\newcommand{\Id}{\operatorname{Id}}

\newcommand{\ben}{\begin{enumerate}}
\newcommand{\een}{\end{enumerate}}

\begin{document}

\title{Semisimple algebras and $PI$-invariants of finite dimensional algebras}

\author{Eli Aljadeff,Yakov Karasik}

\address{Department of Mathematics, Technion - Israel Institute of Technology,
Haifa 32000, Israel}

\email{aljadeff 'at' technion.ac.il (E. Aljadeff),
yaakov 'at' technion.ac.il (Y. Karasik)}

\begin{abstract}
Let $\Gamma$ be a $T$-ideal of identities of an affine PI-algebra over an algebraically closed field $F$ of characteristic zero. Consider the family $\mathcal{M}_{\Gamma}$ of finite dimensional algebras $\Sigma$ with $Id(\Sigma) = \Gamma$. By Kemer's theory it is known that such $\Sigma$ exists. We show there exists a semisimple algebra $U$ which satisfies the following conditions. $(1)$ There exists an algebra $A \in \mathcal{M}_{\Gamma}$ with Wedderburn-Malcev decomposition $A \cong U \oplus J_{A}$, where $J_{A}$ is the Jacobson's radical of $A$ $(2)$ If $B \in \mathcal{M}_{\Gamma}$ and $B \cong B_{ss} \oplus J_{B}$ is its Wedderburn-Malcev decomposition then $U$ is a direct summand of $B_{ss}$. We refer to $U$ as the \textit{unique minimal semisimple algebra} corresponding to $\Gamma$. More generally, if $\Gamma$ is a $T$-ideal of identities of a PI algebra and $\mathcal{M}_{\mathbb{Z}_{2}, \Gamma}$ is the family of finite dimensional super algebras $\Sigma$ with $Id(E(\Sigma)) = \Gamma$. Here $E$ is the unital infinite dimensional Grassmann algebra and $E(\Sigma)$ is the Grassmann envelope of $\Sigma$. Again, by Kemer's theory it is known that such $\Sigma$ exists. Then there exists a semisimple super algebra $U$ with the following properties. $(1)$ There exists an algebra $A \in \mathcal{M}_{\mathbb{Z}_{2}, \Gamma}$ with Wedderburn-Malcev decomposition as super algebras $A \cong U \oplus J_{A}$, where $J_{A}$ is the Jacobson's radical of $A$ $(2)$ If $B \in \mathcal{M}_{\mathbb{Z}_{2}, \Gamma}$ and $B \cong B_{ss} \oplus J_{B}$ is its Wedderburn-Malcev decomposition as super algebras, then $U$ is a direct summand of $B_{ss}$ as super algebras. Finally, we fully extend these results to the $G$-graded setting where $G$ is a finite group. In particular we show that if $A$ and $B$ are finite dimensional $G_{2}:= \mathbb{Z}_{2} \times G$-graded simple algebras then they are $G_{2}$-graded isomorphic if and only if $E(A)$ and $E(B)$ are $G$-graded PI-equivalent.

\end{abstract}

\maketitle

\section{introduction}

Let $F$ be an algebraically closed field of characteristic zero and $F\langle X \rangle$ be the free associative algebra over $F$ on a countable set of variables $X$. Let $\Gamma$ be a $T$-ideal of $F\langle X \rangle$ (i.e. invariant under all algebra endomorphisms of $F\langle X \rangle$). It is easy to see that in fact $\Gamma$ is the ideal of polynomial identities of a suitable associative algebra (e.g. $\Gamma = Id(F\langle X \rangle/\Gamma)$). Kemer's representability theorem says that if $\Gamma\neq 0$, then it is the $T$-ideal of identities of an algebra of the form $E(B)$, the Grassmann envelope of some finite dimensional $\mathbb{Z}_{2}$-graded algebra $B = B_{0} \oplus B_{1}$. Here $E = E_{0} \oplus E_{1}$ is the infinite dimensional unital Grassmann algebra with the natural $\mathbb{Z}_{2}$-grading and $E(B) = E_{0}\otimes B_{0} \oplus E_{1}\otimes B_{1}$ viewed as an ungraded algebra.  In case $\Gamma$ is the $T$-ideal of identities of an affine PI algebra, or equivalently, in case $\Gamma$ contains a nontrivial Capelli polynomial, Kemer's representability theorem says that $\Gamma = Id(A)$ where $A$ is some finite dimensional algebra. Kemer's representability theorem is the key step towards the positive solution of the Specht problem which claims that every $T$-ideal is finitely based.

The purpose of this paper is to prove roughly, that if $A$ is a finite dimensional algebra over an algebraically closed field of characteristic zero $F$, then the maximal semisimple subalgebra of $A$, namely a supplement $A_{ss}$ of the Jacobson's radical $J_{A}$ in $A$, is ``basically uniquely determined'' by $Id(A)$. We show also that a similar result holds for the algebra $E(B)$, that is $\mathbb{Z}_{2}$-graded semisimple supplement of $J_{B}$ in a finite dimensional super algebra $B$ is basically uniquely determined by $\Gamma = Id(E(B))$. Before we state the results precisely, we should remark right away that strictly speaking the semisimple part of a finite dimensional algebra cannot be determined by its $T$-ideal of identities for the simple reason that e.g., $Id(A) = Id(A \oplus A)$. So, by ``basically uniquely determined'' we mean the following.

\begin{theorem}\label{uniquely determined affine}
Let $\Gamma$ be a $T$-ideal of identities and suppose $\Gamma$ contains a Capelli polynomial $c_{n}$, some $n$. Then there exists a finite dimensional semisimple $F$-algebra $U$ that satisfies the following conditions.

\begin{enumerate}

\item
There exists a finite dimensional algebra $A$ over $F$ with $Id(A) = \Gamma$ and such that $A \cong U \oplus J_{A}$ is its Wedderburn-Malcev decomposition.

\item

If $B$ is any finite dimensional algebra over $F$ with $Id(B) = \Gamma$ and $B_{ss}$ is its maximal semisimple subalgebra, then $U$ is a direct summand of $B_{ss}$.

\end{enumerate}
\end{theorem}

Clearly, up to an algebra isomorphism, the semisimple algebra $U$ is unique minimal and is an invariant of $\Gamma$.

Let $A$ be a finite dimensional algebra over $F$ and let $A \cong A_{1} \times \cdots \times A_{q} \oplus J$ be its Wedderburn-Malcev decomposition where $A_{i}$ is simple, $i=1,\ldots,q$, and $J$ is the Jacobson radical.

\begin{definition}\label{full algebra}
We say $A$ is \textit{full} if up to a permutation of the simple components $A_{1}\cdot J \cdot A_{2} \cdots J \cdot A_{q} \neq 0$.

\end{definition}
The following special case of Theorem \ref{uniquely determined affine} will play a key role in the paper.

\begin{theorem}
If two full algebras $A$ and $B$ are PI-equivalent then their maximal semisimple subalgebras are isomorphic. In particular this holds in case $A$ and $B$ are fundamental algebras.
\end{theorem}
For \textit{fundamental} algebras this was proved in \cite{Procesi}. Fundamental algebras are special type of full algebras. They are important in Kemer's theory but will not play a role here (see \cite{AGPR}).

Let us turn now to the case where $\Gamma$ contains no Capelli polynomial. In that case we have the following result.

\begin{theorem}\label{uniquely determined nonaffine}
Let $\Gamma \leq F\langle X \rangle$ be a nonzero $T$-ideal and suppose $\Gamma$ contains no Capelli $c_{n}$, for any $n$. Then there exists a finite dimensional semisimple super algebra $U$ over $F$ which satisfies the following conditions.

\begin{enumerate}

\item

There exists a finite dimensional super algebra $A$ over $F$ with $Id(E(A)) = \Gamma$ and such that $A \cong U \oplus J_{A}$ is its Wedderburn-Malcev decomposition as super algebras.

\item

If $B$ is any finite dimensional super algebra over $F$ with $Id(E(B)) = \Gamma$ and $B_{ss}$ is its maximal semisimple subalgebra, then $U$ is a direct summand of $B_{ss}$ as super algebras.

\end{enumerate}

\end{theorem}

The proof of Theorem \ref{uniquely determined affine} is given in the next section (Section $2$). In Section $3$ we treat the nonaffine case, Theorem \ref{uniquely determined nonaffine}.

In the last two sections of this article we extend the main results to the setting of $G$-graded $T$-ideals and $G$-graded algebras. The main obstacle there is due to the fact that a $G$-graded simple algebra $A$ is not determined up to a $G$-graded isomorphism by the dimensions of the homogeneous components $A_{g}$, $g \in G$.  The proof uses the extension of Kemer's theory to $G$-graded algebras where $G$ is a finite group (see \cite{AB}).

\begin{remark}
The extension of the results above to algebras over fields of finite characteristic and in particular over finite fields does not seem to be straightforward. One of the reasons is that \textit{symmetrization}, an operation which appears in the proofs, may result as zero multiplication. We refer to the work of Belov, Rowen and Vishne on full quivers of representations of algebras over fields of arbitrary characteristic and more generally over commutative Noetherian domains (see \cite{BelROWVISH1}, \cite{BelROWVISH2}, \cite{BelROWVISH3}).  The notion of full quiver is useful for studying the interactions between the radical and the semisimple component of Zariski closed algebras, a notion that appears in Belov's remarkable solution of the Specht problem for affine algebras over fields of finite characteristic (see \cite{BELOV1}). We emphasize that such interactions for Zariski closed algebras are considerably more subtle than for finite dimensional algebras over a field of characteristic zero.
\end{remark}

\begin{section}{preliminaries and the proof of the affine case}\label{Section preliminaries and proof of the affine case}

We start by introducing some combinatorial terminology.

Let $\alpha = (a_{1}, \ldots, a_{q})$ be a $q$-tuple (or multi-set rather since the order will not play a role) of nonnegative integers. For any sub-tuple $\gamma$ of $\alpha$ we let $\sigma(\gamma) = \sum_{a\in \gamma}a$ be the weight of $\gamma$. We set $\sigma(\gamma) = 0$ if $\gamma$ is the empty tuple.

In the sequel the tuple $\alpha$ will correspond to the dimensions of the simple components of a finite dimensional semisimple algebra. More precisely, if $A$ is a finite dimensional algebra over $F$, we let $A \cong A_{1}\times \cdots \times A_{q} \oplus J_{A}$ be its Wedderburn-Malcev decomposition. Then $p_{A} = (dim_{F}(A_{1}), \ldots, dim_{F}(A_{q}))$ is the tuple corresponding to $A$.

\begin{definition}
Let $\alpha = (a_{1}, \ldots, a_{r})$ and $\beta = (b_{1},\ldots, b_{s})$ be tuples of nonnegative integer. We say $\beta$ \textit{covers} $\alpha$ if the tuple $\alpha$ may be decomposed into $s$ disjoint sub-tuples $T_{1}, \ldots, T_{s}$ such that $\sigma(T_{i}) \leq b_{i}$, $i=1,\ldots,s$. In case $r < s$, we extend the tuple $\alpha = (a_{1}, \ldots, a_{r})$ by adding zero's.
\end{definition}

\begin{remark}

\begin{enumerate}
\item

In the definition above we may need to consider empty sub-tuples.

\item

In the sequel it will be convenient to allow tuples with zero's. We view tuples which differ only by the number of zero's as equal.
\end{enumerate}

\end{remark}

\begin{example}
The tuple $(16, 12)$ covers the tuple $(10, 9, 3, 3)$ but it does not cover the tuple $(15, 8, 5)$.
\end{example}

\begin{note}\label{antysymmetric and majorization}
\begin{enumerate}
\item
The covering relation is antisymmetric.

\item
The covering relation is strictly stronger than majorization.
\end{enumerate}
\end{note}

Next we recall some definitions and a result from Kemer's theory.

Let $A$ be a finite dimensional algebra over $F$.  Let $A \cong A_{ss}\oplus J_{A}$ be its Wedderburn-Malcev decomposition where $J_{A}$ is the Jacobson radical and $A_{ss}$ is a semisimple subalgebra supplementing $J_{A}$. The algebra $A_{ss}$ decomposes uniquely (up to permutation) into a direct product of simple algebras $A_{1}\times \cdots \times A_{q}$, where $A_{i} \cong M_{n_{i}}(F)$ is the algebra of $n_{i} \times n_{i}$-matrices over $F$. Furthermore, it is well known all semisimple supplements of $J_{A}$ in $A$ are isomorphic.

It is clear that in order to test whether a multilinear polynomial $p$ is an identity of $A$ it is sufficient to evaluate the polynomial on a basis of $A$ and so we fix from now on a basis $\mathcal{B} = \{e^{i}_{k,l}, u_1,\ldots,u_d\}$. Here, the elements $\{e^{i}_{k,l}\}$, $1 \leq k,l \leq n_{i}$ are the elementary matrices of $M_{n_{i}}(F)$, $i = 1,\ldots, q$, and $\{u_1,\ldots,u_d\}$ is a basis of $J_{A}$.

\begin{definition} Let $p = p(x_{1},\ldots,x_{n})$ be a multilinear polynomial. We say an evaluation of $p$ on $A$ is \textit{admissible} if the variables of $p$ assume values only from the basis $\mathcal{B}$. We refer to an evaluation of a variable as  \textit{semisimple} (resp. \textit{radical}) if the value is an elementary matrix $e^{i}_{k, l}$ (resp. an element $ u_{i} \in J_{A}$).

\end{definition}
\textit{In the sequel we will consider only admissible evaluations.}

\begin{definition}
Let $A$ be a full algebra (Definition \ref{full algebra}). We say a multilinear polynomial $p(x_{1},\ldots,x_{n})$ is
\begin{enumerate}
\item
 $A$\textit{-weakly full} (or \textit{weakly full} of $A$ or \textit{weakly full} when the algebra in question is clear) if it has a nonzero admissible evaluation on $A$ where elements from all simple components are represented in the evaluation.
\item
$A$-\textit{full} if every simple component of $A_{ss}$ is represented in \textit{every} admissible nonzero evaluation on $A$. Also here we may use the terminology \textit{full} of $A$ or just \textit{full}.
\item

$A$-\textit{strongly full} if every basis element of $A_{ss}$ appears in every admissible nonzero evaluation of $p$.

\end{enumerate}
\end{definition}

\begin{remark}
In this paper we make use of polynomials that are \textit{weakly full} or \textit{strongly full}. We mention \textit{full} polynomials here just for completeness. They appear in Kemer's theory (see \cite{AljBelovKar}, Definition $5.10$).
\end{remark}
It is clear that if $p$ is $A$-strongly full then it is full. Also, every full polynomial is weakly full.

We are interested in the opposite direction. We start with
\begin{lemma}\label{full algebras full polynomials}
If $A$ is a full algebra then it admits a weakly full polynomial.
\end{lemma}
\begin{proof}
Let $A$ be as above. Then the multilinear monomial of cardinality $2q-1$ is weakly full. Indeed, we get a nonzero evaluation where we put
$q$ semisimple (resp. $q-1$ radical) values in the odd (resp. even) positions).
\end{proof}

The following theorem is basically \textit{Kemer's Lemma} $1$ (see \cite{AljBelovKar}).

\begin{theorem}\label{full polynomials existence}
The following hold.

\begin{enumerate}

\item
Every full algebra admits a multilinear strongly full polynomial and therefore admits a full polynomial.

\item

Let $A$ be a full algebra and $f_{0}$ be a multilinear weakly full polynomial of $A$. Then there exists a full polynomial $f$ of $A$ in $\langle f_{0}\rangle_{T}$, the $T$-ideal generated by $f_{0}$.

\item

Let $A$ be a full algebra and $f_{0}$ a multilinear weakly full polynomial of $A$. Then there exists a strongly full polynomial $f \in \langle f_{0}\rangle_{T}$ of $A$.

\end{enumerate}
\end{theorem}

\begin{proof}
Clearly, the $3$rd statement implies the $2$nd and together with Lemma \ref{full algebras full polynomials} it implies the $1$st statement. Statement $3$ follows from the construction in the proof of Kemer's lemma $1$ (see \cite{AljBelovKar}).
\end{proof}
As we shall need to refer to the precise construction of strongly full polynomials starting from a weakly full polynomial $f_{0}$, let us recall their construction here. It is convenient to illustrate first the construction on the weakly full polynomial mentioned above.

Let $A\cong A_{1}\times \cdots \times A_{q} \oplus J_{A}$, where $A_{i} \cong M_{n_{i}}(F)$, $i=1,\ldots,q$ (as above) and suppose that after reordering the simple components we have $A_{1}JA_{2}\cdots JA_{q} \neq 0$. Let $f_{0} = X_{1}\cdot w_{1} \cdots w_{q-1}\cdot X_{q}$ be a monomial of $2q-1$ variables which is clearly weakly full by the obvious evaluation. Let
$$Z_{n} = Z_{n}(x_{1},\ldots,x_{n^{2}};y_{1},\ldots,y_{n^{2}+1}) = y_{1}\cdot x_{1}\cdot y_{2} \cdots y_{n^{2}}\cdot x_{n^{2}}\cdot y_{n^{2}+1}$$
be a multilinear monomial on $2n^{2}+1$ variables. For $i = 1,\ldots,q$, we consider $k$ monomials $Z_{n_{i}}$ in disjoint variables, denoted by $Z_{n_{i},l}$, $l= 1,\ldots, k$, where the integer $k$ is sufficiently large and will be determined later. We set $\Delta_{i} = Z_{n_{i},1}\cdots Z_{n_{i},k}$, that is the product of $k$ copies of the monomial $Z_{n_{i}}$ with disjoint sets of variables. Finally, in view of  the inequality  $A_{1}JA_{2}\cdots JA_{q} \neq 0$ we apply the $T$-operation and replace the variable $X_{i}$ by $X_{i}\cdot \Delta_{i}$ in the polynomial $f_{0}$ (here it is just a monomial) and obtain the monomial

$$\Omega =  X_{1}\cdot \Delta_{1}\cdot w_{1}\cdot X_{2}\cdot \Delta_{2}\cdot w_{2} \cdots w_{q-1}\cdot X_{q}\cdot\Delta_{q}.$$

We refer to the $x$'s (lower case) in $\Omega$ as \textit{designated} variables, the $y$'s as \textit{frame} variables and $w$'s as \textit{bridge} variables. Now, it is not difficult to see that the monomial $\Omega$ admits a nonzero evaluation where the $x$'s from $Z_{n_{i},l}$ get values consisting of the full basis of the $i$th simple component, that is the elementary matrices $\{e^{i}_{t,s}\}$,  the $y$'s from $Z_{n_{i},l}$ get values of the form $e^{i}_{t,t}$ and the $w$'s get radical values which bridge the different simple components. Fixing $r = 1, \ldots, k,$ we alternate all $x$'s from the monomials $Z_{n_{i},r}$, $i=1,\ldots,q$, so we obtain $k$ alternating sets of cardinality $dim_{F}(A_{ss})$. We denote the polynomial obtained by $f_{A}$.
We adopt the terminology used in Kemer's theory and refer to each alternating set of designated variables as a \textit{small set}. Moreover, we shall refer to the set of variables $x$ in a small set together with the corresponding frames, that is the $y$ variables that border the $x$ variables, as an \textit{augmented small set}.

 \begin{remark}

In Kemer's theory there is also a notion of \textit{big sets}, sets which involve the alternation of semisimple and bridge variables. We will not make use of big sets here.

\end{remark}
Suppose the integer $k$, namely the number of small sets in $f_{A}$, exceeds the nilpotency index of $A$. Let us show $f_{A}$ is a strongly full polynomial of $A$. We will show that if $\delta$ is any admissible \textit{nonzero} evaluation of $f_{A}$, there is at least one small set which assumes precisely a full basis of $A_{ss}$. Indeed, by the alternation of designated variables we are forced to evaluate each small set on different basis elements and if this is not a full basis of $A_{ss}$, we have that at least one of the designated variables assumes a radical value. Since $k$ is larger than the nilpotency index of $A$, we cannot have a radical evaluation in every small set. This shows $f_{A}$ is strongly full. In fact this proves the last statement of Theorem \ref{full polynomials existence} for the weakly full polynomial $f_{0} = X_{1}\cdot w_{1} \cdots w_{q-1}\cdot X_{q}$.

Let us proceed now to the general case, namely where $f_{0}$ is assumed to be an arbitrary multilinear weakly full polynomial of $A$. Denote by $\Phi$ a nonzero evaluation of  $f_{0}$ which visits every simple component of $A$. Let us denote variables of $f_{0}$ which assume values from the simple components $A_{1},\ldots, A_{q}$ by $X_{1}, \ldots, X_{q}$ respectively. Since the evaluation $\Phi(f_{0})$ is nonzero, it is nonzero on one of the monomials of $f_{0}$ which we fix from now on and denote it by $R_{e}$. We have then that $f_{0} = \sum_{\sigma \in S_{m}} = \lambda_{\sigma}R_{\sigma}$ where $\lambda_{\sigma} \in F$ and $\lambda_{e} = 1$. Here $m$ is the number of variables in $f_{0}$. We proceed now as in the previous case, namely replace the variables $X_{i}$ by $X_{i}\cdot \Delta_{i}$ and obtain a polynomial which we denote by $\Omega$. We have that if $f_{0} = f_{0}(X_{1},\ldots, X_{q}; M)$ then $\Omega = f_{0}(X_{1}\Delta_{1},\ldots, X_{q}\Delta_{q}, M) \in \langle f_{0} \rangle_{T}$. Here $M$ is a suitable set of variables. By an appropriate evaluation of the monomials $\Delta_{i}$, $i=1,\ldots,q$, we see that $\Omega$ is a nonidentity of $A$ and is clearly weakly full. Finally we alternate the designated variables as above and obtain a polynomial which we denote by $f_{A}$. It is not difficult to see that $f_{A}$ satisfies the $3$rd condition of Theorem \ref{full polynomials existence} with respect the given weakly full polynomial $f_{0}$.

\begin{lemma}\label{principal lemma} $($Main Lemma$)$
Notation as above. Suppose $A$ and $B$ are full algebras. Suppose $p_{B}$ does not cover $p_{A}$. Then there exists a strongly full polynomial $f_{A}$ of $A$ which vanishes on $B$. In fact, if $f_{0}$ is any weakly full polynomial of $A$ then there exists a strongly full polynomial $f_{A} \in \langle f_{0} \rangle_{T}$ which vanishes on $B$.
\end{lemma}

\begin{proof}
Let $f_{A}$ be the strongly full polynomial of $A$ as constructed above in case $f_{0} = X_{1}\cdot w_{1} \cdots w_{q-1}\cdot X_{q}$. We take a large number of small sets $k$, exceeding the nilpotency index of $B$. We claim $f_{A}$ is an identity of $B$. We will show that if this is not the case then necessarily $B$ covers $A$. Let us fix a nonzero evaluation $\Phi$ of $f_{A}$ on $B$ and consider one monomial, which we may assume is the monomial $\Omega$ of $f_{A}$ (see the construction above), whose value is nonzero.
Note that by the condition on $k$, there exists an augmented small set, say the $j$th set where $j \in \{1,\ldots,k\}$, which is free of radical values. It follows that the $\Phi$-values of each segment in $\{Z_{n_{1}, j}, \ldots, Z_{n_{q}, j}\}$ consist only of semisimple elements in $B$, and moreover semisimple elements from the same simple component. But because the evaluation of $\Phi$ on $f_{A}$ is nonzero and the variables in the $j$th small set alternate, the semisimple values of $B$ must be linearly independent. This implies that $B$ covers $A$ as desired.

In the general case we may argue as follows. Let $f_{0}$ be an arbitrary weakly full polynomial of $A$ and let $R_{\sigma} = R_{\sigma}(X_{1},\ldots, X_{q}; M)$ be any monomial of $f_{0}$. Applying the $T$-operation on $R_{\sigma}$ we obtain $\Omega_{\sigma} = R_{\sigma}(X_{1}\Delta_{1},\ldots, X_{q}\Delta_{q}, M) \in \langle R_{\sigma} \rangle_{T}$. Next we alternate the designated variables as above and obtain a polynomial which we denote by $(R_{\sigma})_{A}$. As in the first case considered, that is in case where $f_{0} = X_{1}\cdot w_{1} \cdots w_{q-1}\cdot X_{q}$, we see that if $(R_{\sigma})_{A}$ admits a nonzero evaluation on $B$, then $B$ covers $A$. It follows that if $f_{A}$ admits a nonzero evaluation on $B$, this is true also for the polynomial $(R_{\sigma})_{A}$, some $\sigma$, and so $B$ covers $A$.
\end{proof}

\begin{corollary}
Let $A$ and $B$ full algebras. If they are PI-equivalent, then their semisimple parts, $A_{ss}$ and $B_{ss}$ are isomorphic.
\end{corollary}
\begin{proof}
Indeed, $A$ and $B$ must cover each other. It follows that the tuple of dimensions of the simple components of $A$ and $B$ coincide up to a permutation (see Note \ref{antysymmetric and majorization}) and hence they are isomorphic.
\end{proof}

In what follows we will need a somewhat stronger statement.
\begin{corollary}
Let $A$ be a full algebra and $B_{1},\ldots,B_{t}$ be a finite family of full algebras, each not covering $A$. If $f_{0}$ is a weakly full polynomial of $A$ then there is a strongly full polynomial $f_{A} \in \langle f_{0} \rangle$ of $A$ that vanishes on $B_{i}$, $i=1,\ldots,t$. In particular if $B$ is a direct sum of full algebras, each not covering $A$, then there exists a strongly full polynomial $f_{A} \in \langle f_{0} \rangle$ of $A$ which vanishes on $B$.

\end{corollary}

\begin{proof}
We only need to pay attention to the number of small sets $k$ in $f_{A}$, namely it should exceed the nilpotency index of each $J_{B_{i}}$, $i=1,\ldots,t$.
\end{proof}

Recall that any affine PI-algebra $A$ and in particular any finite dimensional algebra is PI-equivalent to a direct sum of full algebras. See for instance \cite{AljBelovKar}, \cite{AGPR}. Here we will need a more precise statement.

\begin{definition}
Let $A$ be finite dimensional algebra. We say $$ Pr(A)_{\{T_{1},\ldots, T_{n}\}}$$ is a \textit{presentation} \textit{of $A$ by full algebras} if the following hold.
\begin{enumerate}

\item

$T_{i}$ is full for $i=1,\ldots,n$.
\item

$Pr(A)_{\{T_{1},\ldots, T_{n}\}} \cong T_{1} \oplus \cdots \oplus T_{n}$.

\item
$Pr(A)_{\{T_{1},\ldots, T_{n}\}}$ is PI equivalent to $A$.

\end{enumerate}
\end{definition}

\begin{remark}\label{fixing the presentation into full subalgebras}
Note that an algebra may have two different presentations which are isomorphic. Thus when referring to a presentation $Pr(A)_{\{T_{1},\ldots, T_{n}\}}$ we are fixing not only the isomorphism type but also the decomposition into full subalgebras (up to permutation). In the sequel, by abuse of notation, we may denote the presentation $Pr(A)_{\{T_{1},\ldots, T_{n}\}}$ just by $Pr(A)$. Note that if $\Gamma$ is a $T$-ideal containing Capelli polynomials we may view $P = P_{\{T_{1},\ldots, T_{n}\}}$ as a presentation of $\Gamma$.
\end{remark}

\begin{proposition}
Let $A$ be finite dimensional. Then there exists a presentation $Pr(A)_{\{T_{1},\ldots, T_{n}\}}$ of $A$ where the semisimple subalgebra $(T_{i})_{ss}$ of $T_{i}$ is a direct summand of $A_{ss}$ for $i=1,\ldots,n$.

\end{proposition}
\begin{proof}
In fact the stronger statement follows from the construction in (\cite{AGPR} subsection $17.2.4.$). Let $A \cong A_{1}\times \cdots \times A_{q} \oplus J_{A}$ be the Wedderburn-Malcev decomposition. Clearly we may assume $A$ is not full. Consider the subalgebra $$\mathcal{A}_{i} = \langle A_{1}, \ldots, A_{i-1}, A_{i+1},\ldots, A_{q}; J_{A}\rangle.$$ We claim $A$ and $\mathcal{A}_{1} \oplus \cdots \oplus \mathcal{A}_{q}$ are PI-equivalent. Clearly $Id(A) \subseteq Id(\mathcal{A}_{1} \oplus \cdots \oplus \mathcal{A}_{q})$. For the converse if $f$ is a nonidentity of $A$, it must be a nonidentity of at least one $\mathcal{A}_{i}$ for otherwise it is a full polynomial of $A$ which implies $A$ is full contrary to our assumption. The proposition is then proved by induction.
\end{proof}

For any presentation $Pr(A)$ of $A$
we let $Pr(A)_{dim(ss)}$ be the set of tuples consisting of the dimensions of the simple components that appear in the different full algebras of $Pr(A)$ and denote by $Pr(A)_{dim(ss), max}$ the set of maximal tuples in $Pr(A)_{dim(ss)}$ with respect to covering.

\begin{corollary}\label{maximal tuples with respect to covering are equal}
The set $Pr(A)_{dim(ss), max}$ depends on $A$ but not on the presentation $Pr(A)$. Hence we can denote the set $Pr(A)_{dim(ss), max}$ by $\mathcal{A}_{dim(ss), max}$.
\end{corollary}

\begin{proof}

Suppose the contrary holds. Let $P_{1}$ and $P_{2}$ be presentations of $A$ as above. Then without loss of generality there exists a full subalgebra $M$ of $P_{1}$ whose tuple is maximal and does not appear as a maximal tuple in $P_{2}$. We may assume $M$ is not covered by tuples of $P_{2}$ for otherwise $M$ is strictly covered by a tuple of $P_{2}$ and in that case we may exchange the roles of $P_{1}$ and $P_{2}$.
Now, by the lemma, there exists a nonidentity polynomial of $M$ which is an identity of every full subalgebra of $P_{2}$ and the claim is proved.

\end{proof}

In the next lemma we show we can \textit{fuse} finite dimensional algebras $A$ and $B$ with isomorphic semisimple subalgebras. More generally, suppose the semisimple subalgebra of $A$ is a direct summand of the semisimple subalgebra of $B$. We claim $A \times B$ is PI-equivalent to an algebra of the form $B_{ss} \oplus \hat{J}$. Yet more generally, suppose $A$ and $B$ have a common semisimple component $U$ (up to isomorphism), then there exists an algebra $C$, $PI$ equivalent to $A \times B$ in which the semisimple algebra isomorphic to $U$ appears in $C$ only once. Here is the precise statement.

\begin{lemma}\label{fusion lemma}
Let $A_{1}\times \cdots \times A_{q} \oplus J(A)$ and $B_{1}\times \cdots \times B_{r} \oplus J(B)$ be the Wedderburn-Malcev decompositions of $A$ and $B$ respectively. Suppose $A_{1} \times \cdots \times A_{k} \cong B_{1} \times \cdots \times B_{k} \cong U$. Then $A \times B$ is PI-equivalent to $C = U \times A_{k+1} \times \cdots \times A_{q} \times B_{k+1} \times \cdots \times B_{r} \oplus J(A) \oplus J(B)$

\end{lemma}
\begin{proof}
We consider the vector space embedding 
$$C = U \times A_{k+1} \times \cdots \times A_{q} \times B_{k+1} \times \cdots \times B_{r} \oplus J(A) \oplus J(B)$$ $$\hookrightarrow  [U\times  A_{k+1} \times \cdots \times A_{q} \oplus J(A)]\times [U\times B_{k+1}\times \cdots \times B_{r} \oplus J(B)]$$
where the elements of $U$ are mapped diagonally. It is easy to see the image is closed under multiplication yielding an algebra structure on $C$. As for the polynomial identities the above embedding (now, as algebras) yields $Id(C) \supseteq Id(A \times B) = Id(A) \cap Id(B)$. On the other hand the algebras $A$ and $B$ are embedded in $C$ and the result follows.

\end{proof}

Let $\Gamma$ be a $T$-ideal as above and fix a presentation $P_{1}$ of $\Gamma$. We shall present a sequence of operations on $P_{1}$ which will lead to the proof of Theorem \ref{uniquely determined affine}. Note that Theorem \ref{uniquely determined affine} is clear if $\Gamma = Id(A)$ where $A$ is nilpotent. So we shall assume whenever it is necessary that $A$ is non-nilpotent.

Let $P_{1}$ and $P_{2}$ be presentations of $A$ and let $P_{1, ext}$ and $P_{2, ext}$ be the corresponding direct sum of full algebras whose semisimple subalgebra is maximal with respect to \textit{covering}. Following the lemma above we may fuse full algebras $T_{1}$  and $T_{2}$ with isomorphic semisimple parts $(T_{1})_{ss} \cong (T_{2})_{ss} \cong U$, where $T_{1}$ appears in $P_{1, ext}$ and $T_{2}$ appears in $P_{2, ext}$. More precisely (notation as above).

Let $P_{1}$ and $P_{2}$ be presentations of $A$ and suppose $T_{1} \cong (T_{1})_{ss} \oplus J_{T_{1}}$ (resp. $T_{2} \cong (T_{2})_{ss} \oplus J_{T_{2}}$) is a full subalgebra of $P_{1}$ (resp. $P_{2}$). Suppose $(T_{1})_{ss} \cong T_{2})_{ss} \cong U$ and replace $T_{1} \cong U \oplus J_{T_{1}}$ by $T'_{1} = U \oplus J_{T_{1}} \oplus J_{T_{2}}$. We note that $T'_{1}$ is full (since $T_{1}$ is) and denote by $P'_{1}$ the corresponding decomposition into full algebras.

\begin{lemma}\label{PI equivalence after fusion}
The algebra $P'_{1}$ is a presentation of $A$.
\end{lemma}
\begin{proof}
Because $Id(T_{1}), Id(T_{2}) \supseteq Id(A)$ we have that $Id(P'_1) \supseteq Id(A)$. On the other hand $Id(P'_1) \subseteq Id(P_{1})$ ($= Id(A))$ and the result follows.
\end{proof}

\begin{remark}
Note that fusion of fundamental algebras $A$ and $B$ with isomorphic semisimple subalgebras yields a fundamental algebra (see \cite{AGPR} for the definition of fundamental algebras).
\end{remark}

\end{section}

Let $\Gamma$ be the $T$-ideal of identities of a finite dimensional algebra. Denote by $\mathcal{M}_{\Gamma}$ the family of presentations $A = A_{\{T_{1},\ldots, T_{n}\}}$ of $\Gamma$.

We perform the following operations on $A \in \mathcal{M}_{\Gamma}$.

\begin{definition}\label{partitions realized}
We let $A_{part}$ be the multi-set (i.e. repetitions are allowed) of unordered tuples realized by the dimensions of the simple components of semisimple subalgebras of the full algebras appearing in the decomposition of $A$. Alternatively, we may think of $A_{part}$ as the set of semisimple algebras appearing in the full algebras, summands of $A$.
\end{definition}

Step ${0}$ (deletion): Let $A\in \mathcal{M}_{\Gamma}$. We delete from $A$ full algebras that do not alter $Id(A)$. Let $A_{i}$ be a full subalgebra of $A$. Denote by $\widehat{A_{i}}$ the summand of $A$ consisting the direct sum of full algebra $A_{j}$, $j \ne i$. Then, we delete $A_{i}$ from the direct sum if $Id(A_{i}) \supseteq Id(\widehat{A_{i}}) = \cap_{j \ne i} Id(A_{j})$. We abuse notation and simply write the outcome by $F_{0}(A)$, an operation of type $0$ on $A$, although the operation depends on the choice of the full algebra $A_{i}$. We write $A = A_{red_0}$ if \textit{every} operation of type $0$ on $A$ is the identity.

Step ${1}$ (fusion): $A\in \mathcal{M}_{\Gamma}$ and suppose $A = A_{red_0}$. We fuse full subalgebras with isomorphic semisimple subalgebras. More generally, if $A_{i}$ and $A_{j}$ $i \neq j$, are full subalgebras of $A$ and $(A_{i})_{ss}$ is a direct summand of $(A_{j})_{ss}$, then the operation $F_{1} = (F_{1})_{A_{i}, A_{j}}$ on $A$ is the fusion of $A_{i}$ and $A_{j}$. We abuse notation and simply write the outcome by $F_{1}(A)$, an operation of type $1$ on $A$, although the operation depends on the choice of the full algebras $A_{i}$ and $A_{j}$. We write $A = A_{red_{0,1}}$ if every operation of type $0$ or $1$ on $A$ is the identity.

We come now to a step where we \textit{decompose} full algebras.

Step $2$ (decomposition): Let $A\in \mathcal{M}_{\Gamma}$ and suppose that $A = A_{red_{0,1}}$. We define an operation of type $2$ on $A$, denoted by $F_{2}$, as follows. Choose a full algebra $Q$ appearing in the decomposition of $A$ into full algebras and let $A_{supp(Q)} = (\hat{Q})_{1} \oplus \cdots \oplus (\hat{Q})_{n}$ be the supplement of $Q$ in $A$. Note that since $A = A_{red_{0,1}}$ there is no full algebra component of $A_{supp(Q)}$ with semisimple part $\cong Q_{ss}$. Suppose there exists a weakly full polynomial $p$ of $Q$ which vanishes on $A_{supp(Q)}$.
In that case we leave the algebra $A$ unchanged, that is $F_{2}(A) = A$.
Otherwise we proceed as follows:

Clearly $Q$ is not nilpotent because $A = A_{red_{0,1}}$. Let us treat the case where $Q_{ss}$ is simple separately. If $Q_{ss}$ is simple and every weakly full polynomial of $Q$ is a nonidentity of $A_{supp(Q)}$ we claim $Id(A) = Id(A_{supp(Q)} \oplus J_{Q})$ where $J_{Q}$ is the radical of $Q$. It is clear that $Id(A) \subseteq Id(A_{supp(Q)} \oplus J_{Q})$. Conversely, suppose $p$ is a nonidentity of $A$. If $p$ is a nonidentity of $A_{supp(Q)}$ it is also a nonidentity of $Id(A_{supp(Q)} \oplus J_{Q})$ as needed, so let us assume $p$ is an identity of $A_{supp(Q)}$. In that case $p$ must be a nonidentity of $Q$ and an identity of $Q_{ss}$. It follows that $p$ is a nonidentity of $J_{Q}$ and we are done.
Let $Q \cong \Delta_{1} \times \cdots \times \Delta_{q} \oplus J_{Q}$ be the Wedderburn-Malcev decomposition of $Q$, where now $q > 1$.
We are assuming every weakly full polynomial of $Q$ is a nonidentity of $A_{supp(Q)}$. In that case we claim the following.

\begin{claim} We can replace the full subalgebra $Q$ of $A$ by a direct sum of full subalgebras $\mathcal{Q}_{1} \oplus \ldots \oplus \mathcal{Q}_{q}$, where for each $i = 1,\ldots,q$, the semisimple algebra $(\mathcal{Q}_{i})_{ss}$ is a proper summand of $Q_{ss}$ $($in particular strictly covered by $Q$$)$ and if $\bar{A}$ denotes the algebra obtained, we have $Id(\bar{A}) = Id(A) = \Gamma$.
\end{claim}

\begin{proof}
Consider the algebras $\mathcal{Q}_{i}$, $i=1,\ldots,q$, obtained from $Q$ by deleting one simple component $\Delta_{i}$ and keeping the radical unchanged. We claim
$A$ is PI-equivalent to $A_{supp(Q)} \oplus \mathcal{Q}_{1} \oplus \cdots \oplus \mathcal{Q}_{q}$. Indeed, it is clear that every identity of the former algebra vanishes on the latter one. Conversely, let $p$ be a nonidentity of the former one. We show it does not vanish on the latter. Clearly, we may assume $p$ vanishes on $A_{supp(Q)}$
and so, by assumption $p$ is not a weakly full polynomial of $Q$. This means that $p$ has no nonzero evaluation on $Q$ which visits all simple components of $Q$ and being a nonidentity of $Q$ it must be a nonidentity of $\mathcal{Q}_{i}$ for some $i$ and hence a nonidentity of the latter.
\end{proof}

We write $A = A_{red_{0,1,2}}$ if any operation of type $0$, $1$ or $2$ on $A$ is the identity.

Similarly to our notation for the operations $F_{0}$ and $F_{1}$ above we abuse notation here and simply write $F_{2}(A) = F_{2, Q}(A)$.

Step ${3}$ (absorption): Fix a presentation $A \in \mathcal{M}_{\Gamma}$ and suppose $A = A_{red_{0,1,2}}$. Let $B \in \mathcal{M}_{\Gamma}$. Roughly speaking the operation
$F^{cond}_{3}$ consists in replacing a full subalgebra $Q$ of $A$ with the fusion of $Q$ with certain full subalgebras of $B$. More precisely, choose a full subalgebra $Q$ of $A$ and a full subalgebra $V$ of $B$ such that $V_{ss}$ is a direct summand of (possibly isomorphic to) $Q_{ss}$. Then replace the full subalgebra $Q$ in $A$ by the fusion of $Q$ and $V$. We denote the outcome by $(F^{cond}_{3})_{B, Q, V}(A)$ or simply by $(F^{cond}_{3})(A)$. The superscript \textit{cond} means that this operation is conditional. We define $(F_{3})_{B, Q, V}(A)$ as follows. Let $A^{cond} = (F^{cond}_{3})_{B, Q, V}(A)$. If $A^{cond} = (A^{cond})_{red_{0,1,2}}$, we set $(F_{3})_{B, Q, V}(A) = A$, otherwise we set $(F_{3})_{B, Q, V}(A) = A^{cond}$.

\begin{remark}
The point for introducing the conditional operation is that we want an operation of type $3$ to be nontrivial only if operation $0$, $1$ or $2$ have a real effect on $A^{cond}$. This is to prevent the radical grows indefinitely.
\end{remark}

We write $A = A_{red_{0,1,2,3}}$ if any operation of type $0$, $1$, $2$ or $3$ on $A$ is the identity.

Let us describe now the procedure applied to $A \in \mathcal{M}_{\Gamma}$.
\begin{enumerate}
\item
Apply operations of type $0$ on $A$ until any additional operation of type $0$ acts as an identity. Denote the outcome by $A'$.
\item
If there exists an operation of type $1$ with $F_{1}(A') \ne A'$, we apply $F_{1}$ on $A'$ and return to step $0$ with $A := F_{1}(A')$. We continue until we get an algebra $A''$ such that $F_{\epsilon}(A'') = A''$, $\epsilon = 0,1$.

\item
If there exists an operation of type $2$ with $F_{2}(A'') \ne A''$, we apply $F_{2}$ on $A''$ and return to step $0$. We continue until we get an algebra $A'''$ such that $F_{\epsilon}(A''') = A'''$, $\epsilon = 0,1,2$.

\item
If there exists an operation of type $3$ with $F_{3}(A''') \ne A'''$, we apply $F_{3}$ on $A'''$ and return to step $0$. We continue until we get an algebra $A''''$ such that $F_{\epsilon}(A'''') = A''''$, $\epsilon = 0,1,2,3$.

\end{enumerate}

\begin{theorem}\label{algebra after steps 0-3}

For every presentation $A \in \mathcal{M}_{\Gamma}$ the process above stops. In particular, given a presentation $A$, applying operations of type $0-3$ we obtain a presentation $\mathcal{A}\in \mathcal{M}_{\Gamma}$ such that $\mathcal{A} = \mathcal{A}_{red0,1,2,3}$.

\end{theorem}
\begin{proof}

For the proof we shall introduce the following \textit{numerical counters}.
\begin{enumerate}
\item

Let $A\in \mathcal{M}_{\Gamma}$. We denote by $r_{A}$ the number of full subalgebras in the presentation of $A$.

\item

Let $A\in \mathcal{M}_{\Gamma}$ and let $A_{part}$ be the corresponding multi-set of unordered tuples (see Definition \ref{partitions realized}). If $\sigma = (\sigma_{1}, \ldots, \sigma_{r}) \in A_{part}$ is a tuple, we let $n_{\sigma} = 2^{r^2}\sum_{i}\sigma_{i}$ be the weight of $\sigma$. Note that the function $f(r) = 2^{r^2}$ satisfies the condition $(r-1)f(r-1) < f(r)$, a condition that will be used later. We denote $n_{A} = n_{A_{part}} = \sum_{\sigma \in A_{part}}n_{\sigma}$ the \textit{weight} of $A$. Let $d$ be an arbitrary (nonnegative) integer. Observe that the number of tuples $\sigma$ with $n_{\sigma} \leq d$ is finite and hence the number of multi-sets $A_{part}$ corresponding to $A\in \mathcal{M}_{\Gamma}$ with $n_{A_{part}} \leq d$ is finite as well.

\end{enumerate}
\begin{lemma}
The following hold.
\begin{enumerate}

\item
Let $A \in \mathcal{M}_{\Gamma}$ and let $\bar{A} = F_{\epsilon}(A)$, $\epsilon = 0,1$. If $\bar{A} \neq A$ then $r_{\bar{A}} < r_{A}$ and $n_{\bar{A}} \leq n_{A}$.
\item

Let $A \in \mathcal{M}_{\Gamma}$ and suppose $A = A_{red  0, 1}$. Let $\bar{A} = F_{2}(A)$. If $\bar{A} \neq A$ then $n_{\bar{A}} < n_{A}$.

\end{enumerate}

\end{lemma}

\begin{proof}
The first part is clear since we are suppressing a full subalgebra of the presentation of $A$. Note that if we are suppressing a nilpotent algebra $n_{\bar{A}} = n_{A}$. For the proof of the second part let $A = A_{red 0, 1}\in \mathcal{M}_{\Gamma}$. This implies no full subalgebras of $A$ are nilpotent unless $A$ is nilpotent, a case we have already addressed (see paragraph above Lemma \ref{PI equivalence after fusion}). Suppose $F_{2}(A) \neq A$. This means that one tuple $\sigma = (\sigma_{1},\ldots, \sigma_{m})$, $m \geq 1$ is replaced by $m$ tuples each of which has length $m-1$ and is obtained from $\sigma$ by deleting $\sigma_{i}$, $i = 1,\ldots,m$. It follows that the the quantity $2^{(m^{2})}\sum_{i}\sigma_{i}$, the contribution of $\sigma$ to $n_{A}$, is replaced by $(m-1)2^{((m-1)^{2})}\sum_{i}\sigma_{i}$. As $(m-1)2^{((m-1)^{2})} <  2^{(m^{2})}$, the result follows.
\end{proof}

In order to complete the proof of the theorem we consider the pairs $\Theta_{A} = (n_{A}, r_{A})$, $A \in \mathcal{M}_{\Gamma}$ with the lexicographic order $\preceq$ (and $\prec$ if the inequality is strict). Let $\bar{A} = F_{\epsilon}(A)$, $\epsilon = 0,1,2$. It follows that if $\bar{A} \neq A$, invoking the lemma above, we have $\Theta_{\bar{A}} \prec \Theta_{A}$. This completes the proof of the Theorem.

\end{proof}

\begin{corollary}\label{final result after operations}
Given a presentation $A \in \mathcal{M}_{\Gamma}$, the application of steps $0-3$ to $A$ yields a presentation $\bar{A} \in \mathcal{M}_{\Gamma}$ with the following properties.

\begin{enumerate}

\item
If $Q$ is any full subalgebra of $\bar{A}$ then there exists a full subalgebra $V$ of $A$ such that $Q_{ss}$ is a direct summand of $V_{ss}$.

\item

If $Q$ is a full subalgebra of $\bar{A}$, there is a strongly full polynomial of $Q$ which vanishes on the supplement of $Q$ in $A$.

\item

If $Q$ is a full subalgebra of $\bar{A}$, and $B \in \mathcal{M}_{\Gamma}$, then there is a strongly full polynomial of $Q$ which vanishes on every full algebra $V$ of $B$ whose semisimple subalgebra $V_{ss}$ strictly covers $Q_{ss}$ and appears as a summand of the semisimple subalgebra of a full subalgebra of $\bar{A}$.
\end{enumerate}
\end{corollary}

\begin{proof}
By Theorem \ref{algebra after steps 0-3} we may assume $\bar{A} = \bar{A}_{red0,1,2,3}$  The operations of type $0$ and $1$ suppress full algebras of $A$ whereas in operation $2$ we decompose the semisimple part of a full algebra $Q$ into the direct sum of full algebras whose semisimple part is a direct summand of $Q_{ss}$. This proves the first statement. Also the second statement follows easily from the construction. Indeed, if this is not the case there is an operation of type $2$ which is not the identity on $\bar{A}$ contradicting $\bar{A} = \bar{A}_{red0,1,2,3}$.

Let us prove the last statement. By the claim we have that if such polynomial does not exist for a suitable full subalgebras $V$ of an algebra $B \in \mathcal{M}_{\Gamma}$, fusion of $V$ with the corresponding full algebras of $\bar{A}$ generates a decomposition of $Q$ into full algebras whose semisimple algebra is a strict summand of $Q_{ss}$. This contradicts $\bar{A} = \bar{A}_{red0,1,2,3}$ and the result follows.

\end{proof}

\begin{remark}
Note that it is possible that a presentation $B \in \mathcal{M}_{\Gamma}$ contains a full algebra $V$ whose semisimple part $V_{ss}$ does not appear as a direct summand of a full algebra of $\bar{A}$. This does not contradict the last statement of Theorem \ref{algebra after steps 0-3}.
\end{remark}

Let $\bar{A}$ be the algebra obtained from $A$ as in the theorem above and let $\bar{A} = T_{1}\oplus \ldots \oplus T_{n}$ be its decomposition into the direct sum of full algebras. Let $\bar{A}_{part}$ be the set of semsimple algebras appearing in $\bar{A}$, that is $\bar{A}_{part} = \{(T_{i})_{ss}\}_{i=1,\ldots,n}$ (see Definition \ref{partitions realized}).

Our goal is to show $\bar{A}_{part}$ is uniquely determined by $\Gamma$. More precisely

\begin{theorem}\label{equal decomposition full}

If $A, B \in \mathcal{M}_{\Gamma}$ then $\bar{A}_{part} = \bar{B}_{part}$.
\end{theorem}
\begin{remark}
Note that we know the result for maximal points where $A, B \in \mathcal{M}_{\Gamma}$ are arbitrary (see Corollary \ref{maximal tuples with respect to covering are equal}).
\end{remark}
\begin{proof}
Suppose the theorem is false and consider the family $\Omega$ of all full subalgebras of $\bar{A}$ (resp. $\bar{B}$) whose semisimple part does not appear in $\bar{B}$ (resp. $A$). Let $Q \in \Omega$ be maximal with respect to covering and assume without loss of generality that $Q = Q_{\bar{A}}$ is a full subalgebra of $\bar{A}$. Now, by the maximality of $Q_{\bar{A}}$ the semisimple part of every full subalgebra of $\bar{B}$ that strictly covers $Q_{\bar{A}}$ appears in $\bar{A}$. It follows, by Corollary\ref{final result after operations}($3$), there exists a full polynomial $p$ which vanishes on every full subalgebra of $\bar{B}$ that strictly covers $Q_{\bar{A}}$. Furthermore, by our construction of strongly full polynomials there exists such $p$ that vanishes on every full subalgebra of $\bar{B}$ that does not cover $Q_{\bar{A}}$ and so $p$ vanishes on $\bar{B}$. This contradicts $\bar{A}$ and $\bar{B}$ are PI equivalent and the theorem is proved.

\end{proof}

Step ${4}$ (merging): In the final step we merge full subalgebras. Let $\bar{A}$ be an algebra as in the theorem. For each isomorphism type of a simple algebra $M_{n}(F)$ we let $d_{n}$ be the maximal appearance of $M_{n}(F)$ in a full subalgebra of $\bar{A}$. Then we let $\mathcal{A}_{\Gamma, ss} = \Lambda_{n_1} \oplus \cdots \oplus \Lambda_{n_t}$ where $\Lambda_{n_i}$ is the direct sum of $d_{n_i}$ copies $M_{n_i}(F)$. Finally we let $\mathcal{A} \cong \mathcal{A}_{\Gamma, ss} \oplus J_{A}$, where the direct sum is of vector spaces.

\begin{theorem}
There is exists an algebra structure on $\mathcal{A}_{\Gamma}$ so that
\begin{enumerate}
\item
$Id(\mathcal{A}_{\Gamma}) = \Gamma$.

\item
If $B$ is finite dimensional and $\Id(B) = \Gamma$ then $\mathcal{A}_{\Gamma, ss}$ is isomorphic to a direct summand of $B_{ss}$.

\end{enumerate}

\end{theorem}
\begin{proof}
For the algebra structure on $\mathcal{A}$ we set the product as follows. The product on $\mathcal{A}_{\Gamma, ss}$ is already determined. Products of radical elements which belong to different full algebras is set to be zero. Let us determine the multiplication of semisimple elements with radicals. Using distributivity we let $z\in J_{\bar{A}_{i}}$ where $\bar{A}_{i}$ is a full summand of $\bar{A}$. Choose a summand of $(U_{i})_{ss}$ of $\mathcal{A}_{\Gamma, ss}$ isomorphic to $(\bar{A}_{i})_{ss}$. Let $K$ be the semisimple supplement of $(U_{i})_{ss}$ in $\mathcal{A}_{\Gamma, ss}$, that is $(U_{i})_{ss} \oplus K \cong \mathcal{A}_{\Gamma, ss}$. Then we set the product of $z$ with semsimple elements of $(U_{i})_{ss}$ as in $\bar{A}_{i}$ whereas the multiplication of $z$ with elements of $K$ is set to be zero. Let us show $Id(\mathcal{A}) = \Gamma$. Each $\bar{A}_{i}$ is isomorphic to a summand of $\mathcal{A}$ and so $Id(\bar{A}) \supseteq Id(\mathcal{A})$. For the opposite inclusion let $p$ be a multilinear nonidentity of $\mathcal{A}$ and fix a nonzero evaluation on $\mathcal{A}$. Since the multiplication of radical elements of different summands $J_{\bar{A}_{i}}$ and $J_{\bar{A}_{j}}$ is zero the evaluation may involve at most radicals from $J_{\bar{A}_{i}}$, for a unique $i$. For that $i$, semisimple elements that appear in the evaluation must belong to the summand $(U_{i})_{ss}$. We see the polynomial $p$ is a nonidentity of $\bar{A}_{i}$ and so a nonidentity of $\bar{A}$.
For the proof of the second statement, by the construction of $\mathcal{A}_{\Gamma}$ from $\bar{A}$ we see $\mathcal{A}_{\Gamma, ss}$ is a direct summand of $\bar{A}_{ss}$ and hence, by Theorem \ref{equal decomposition full}, also of $\bar{B}_{ss}$. Furthermore, we see from Step $4$ that every $\Lambda_{n_i}$ is a direct summand of the semisimple part of a full summand of $\bar{A}$ and hence of $\bar{B}$. We complete the proof of the theorem invoking Corollary \ref{final result after operations}($1$).

\end{proof}
\section{Nonaffine algebras}\label{nonaffine algebras section}

In this section we prove Theorem \ref{uniquely determined nonaffine}.

Note that the key point in the construction of strongly full polynomials for a finite dimensional full algebra $A$ was the fact that in any nonzero evaluation we were forced to evaluate the designated variables in at least one small set by a complete basis of semisimple elements. Then, for such polynomial we showed it is an identity of any full algebra $B$ that does not cover $A$. Now, if $A$ is a finite dimensional full super algebra (see \cite{AB}), it is not difficult to construct a \textit{super} strongly full polynomial with a similar property, that is, a polynomial $p$ that visits a full basis of the semisimple part of $A$ in every nonzero evaluation. However, this is not what we need. For the proof, we need an \textit{ungraded} polynomial, nonidentity of $E(A)$, which visits the different super simple components of $A$
in any nonzero evaluations of the form $\epsilon \otimes u$. Here, $\epsilon = 1 \in E$ or $\epsilon_{i} \in E$ is a generator, and $u \in A$. Furthermore, as in the affine case, we shall need a full basis $\{u\} \subseteq A_{ss}$ to appear in every nonzero evaluation of $f$.
In fact, as in the affine case, we will need to construct such polynomials for $E(A)$ that belong to the $T$-ideal generated by an arbitrary weakly full polynomial of $E(A)$.

Once we have constructed such polynomials for $E(A)$ where $A$ is a finite dimensional full super algebra, we will be able to show the analogue of the Main lemma. The proof of Theorem \ref{uniquely determined nonaffine} will then follow the same lines as the proof in the affine case.

We start by defining a partial ordering on finite dimensional semisimple $\mathbb{Z}_{2}$-graded algebras.

Let $A = A_{1} \oplus \cdots \oplus A_{q}$ and $B = B_{1} \oplus \cdots \oplus B_{s}$ be the decompositions of semisimple algebras $A$ and $B$ into direct sum of finite dimensional $\mathbb{Z}_{2}$-graded simple algebras $A_{i}$ and $B_{j}$ respectively. Consider the pair $p_{A} = (p_{A, 0}, p_{A,1})$ where $p_{A, 0} = (a_{0,1}, \ldots, a_{0,q})$ and $p_{A,1} = (a_{1,1}, \ldots, a_{1,q})$ are $q$-tuples consisting the dimensions of the $0$-components and the $1$-components of the $\mathbb{Z}_{2}$-graded simple summands of $A$. Similarly we have the pair $p_{B} = (p_{B, 0}, p_{B,1})$ and $t$-tuples $p_{B, 0} = (b_{0,1}, \ldots, b_{0,t})$ and $p_{B,1} = (b_{1,1}, \ldots, b_{1,q})$ for the algebra $B$.
\textbf{}
\begin{definition}
We say $B$ covers $A$ (or $p_{B}$ covers $p_{A}$) if there exists a decomposition of the tuple $(1,\ldots, q)$ into $t$ subsets such that the sum of the elements of $p_{A, 0} = (a_{0,1}, \ldots, a_{0,q})$ corresponding to the $i$th subset is bounded by $b_{0,i}$ and the corresponding sum of odd elements in  $p_{A,1} = (a_{1,1}, \ldots, a_{1,q})$ is bounded by $b_{1,i}$ (same $i$), $i=1,\ldots,t$.
\end{definition}

\begin{example}
Consider the pair of tuples $p_{B} = (p_{B_{0}}, p_{B_{1}})$ where $p_{B_{0}} = (17, 13)$ and $p_{B_{1}} = (8,12)$. It covers the pair $p_{A} = (p_{A_{0}}, p_{A_{1}})$ where $p_{A_{0}} = (16, 10, 2)$ and $p_{A_{1}} = (0, 4, 2)$. On the other hand the pair $p_{B} = (p_{B_{0}}, p_{B_{1}})$ where
$p_{B_{0}} = (17, 13)$ and $p_{B_{1}} = (8,12)$ does not cover the pair $p_{A} = (p_{A_{0}}, p_{A_{1}})$ where
$p_{A_{0}} = (10, 10, 4)$ and
$p_{A_{1}} = (6, 6, 4)$. Note, however, that the tuple $(17,13)$ (resp. $(8,12)$) does cover $(10,10,4)$ (resp. $(6,6,4)$).
\end{example}

Let $A$ be a finite dimensional super algebra over an algebraically closed field $F$ of characteristic zero. Let $A \cong A_{ss} \oplus J$ be the Wedderburn-Malcev decomposition of $A$. Let $A_{ss} \cong A_{1} \times \cdots \times A_{q}$ where $A_{i}$ are super simple algebras.
\begin{definition}
We say $A$ is full if up to ordering of the super simple components we have $A_{1}\cdot J\cdot A_{2}\cdots J\cdot A_{q} \neq 0$.
\end{definition}

Before stating the Main Lemma, let us make precise definitions of admissible evaluations of polynomials as well as weakly full, full and strongly full polynomials of $E(A)$ where $A$ is a finite dimensional full super algebra.

Let $U$ be a finite dimensional $\mathbb{Z}_{2}$-simple algebra. It is well known that $U$ is either isomorphic to $M_{l,f}(F)$ where the grading is elementary and is determined by an $(l+f)$-tuple with $l$ $e$'s and $f$ $\sigma$'s, where an elementary matrix $e_{i,j}$ has degree $e$ if $1 \leq i,j \leq l$ or $l+1 \leq i,j \leq l+f$ and degree $\sigma$ otherwise or isomorphic to $FC_{2} \otimes M_{n}(F)$ where elements of the form $u_{e}\otimes e_{i,j}$ have degree $e$ and elements of the form $u_{\sigma}\otimes e_{i,j}$ have degree $\sigma$. Note that the set $\{e_{i,j}\}$ (resp. \{$u_{e}\otimes e_{i,j}\}$) is a basis of $M_{l,f}(F)$ (resp. of $FC_{2} \otimes M_{n}(F)$). We denote by $\beta_{ss}$ a basis of $A_{ss}$ of elements of all elements of that form. Note that all basis elements in $\beta_{ss}$ are homogeneous. If $U$ is any simple component of $A_{ss}$, and $z$ denotes a basis element of $U$ as above we consider a basis $\Sigma_{ss}$ of $E(A_{ss})$ consisting of all elements of the form $\epsilon_{i_1}\cdots\epsilon_{i_n} \otimes z$, $n$ is even (in case $n=0$, we set $\epsilon_{i_1}\cdots\epsilon_{i_n} = 1$) and $z \in \beta_{ss}$ has degree $e$ or $\epsilon_{i_1}\cdots\epsilon_{i_n} \otimes z$, $n$ is odd and $z\in \beta_{ss}$ has degree $\sigma$. Here $\epsilon_{i_1},\ldots, \epsilon_{i_n}$ are different generators of the Grassmann algebra $E$. Finally, we choose an \textit{homogeneous} basis $\beta_{J}$ of the Jacobson radical $J$ of $A$ and consider a basis $\Sigma_{J}$ of $E(J)$ consisting of all elements of the form $\epsilon_{i_1}\cdots\epsilon_{i_n} \otimes w$ where (as above) $n$ is even and $w \in \beta_{J}$ is of degree $e$ or $n$ is odd and $w \in \beta_{J}$ is of degree $\sigma$.
\begin{definition}
Let $p$ be a multilineal polynomial. We say an evaluation of $p$ on $E(A)$ is admissible if all values are taken from $\Sigma_{ss}$ or $\Sigma_{J}$.
\end{definition}

\begin{definition}
Let $A$ be a finite dimensional full super algebra as above.
\begin{enumerate}
\item

We say a multilinear polynomial $p$ is \textit{weakly full} of $E(A)$ if there is an admissible nonzero evaluation of $p$ on $E(A)$ where among the elements $\epsilon_{i_1}\cdots\epsilon_{i_n} \otimes z$, $z\in A_{ss}$ that appear in the evaluation, we have at least one elements $z$ from each $\mathbb{Z}_{2}$-simple component of $A_{ss}$.
\item
We say a multilinear polynomial $p$ is \textit{full} of $E(A)$ if all $\mathbb{Z}_{2}$-simple subalgebras of $A_{ss}$ appear in every nonzero admissible evaluation of $p$ on $E(A)$. That is for every $Z_{2}$-simple component $A_{i}$, $i=1,\ldots,q$, there is a variable of $p$ whose value is of the form $\epsilon_{i_1}\cdots\epsilon_{i_n} \otimes z$ for some $z \in A_{i}$.
\item
We say a multilinear polynomial $p$ is \textit{strongly full} of $E(A)$ if for every nonzero admissible evaluation of $p$ on $E(A)$ and every $z \in A_{ss}$, there is variable of $p$ whose value is of the form $\epsilon_{i_1}\cdots\epsilon_{i_n} \otimes z$.
\end{enumerate}
\end{definition}

The following statement is the Main Lemma in the nonaffine case.
\begin{lemma}
Suppose $A$ and $B$ are finite dimensional $\mathbb{Z}_{2}$-graded full algebras. Suppose $p_{B}$ does not cover $p_{A}$. Then there exists a strongly full polynomial nonidentity of $E(A)$ which is an identity of $E(B)$. Furthermore if $f$ is an arbitrary weakly full polynomial of $E(A)$, then there exists a strongly full polynomial $p \in \langle f \rangle_{T}$ of $E(A)$.
\end{lemma}
\begin{proof}

As in the affine case we have that a monomial $f_{0} = X_{1}\cdot w_{1} \cdots w_{q-1}\cdot X_{q}$ of degree $2q-1$ is weakly full of $A$ (that is $p$ visits every $Z_{2}$-simple component of $A$) and hence weakly full of $E(A)$. We proceed with the construction of a strongly full polynomial in $\langle f_{0} \rangle_{T}$.
\end{proof}
Let $d_{0}$ (resp. $d_{1}$) be the dimension of the even (resp. odd) homogeneous component of $A_{ss}$.
We consider a diagram composed of two strips of semisimple elements, denoted by $\alpha_{i,j}$ and similarly two strips of variables $x_{i,j}$, horizontal and vertical,  where the horizontal strip has $d_{0}$ rows and $k$ columns and the vertical strip has $d_{1}$ columns and $k$ rows ($k$ to be determined).

\[
\begin{array}{cc}
 & \begin{array}{|cccc|}
\hline \alpha_{1,d_{1}+1} & \alpha_{1,d_{1}+2} & \cdots & \alpha_{1,d_{1}+k}\\
\alpha_{2,d_{1}+1} & \alpha_{2,d_{1}+2} & \cdots & \alpha_{2,d_{1}+k}\\
\vdots & \vdots &  & \vdots\\
\alpha_{d_{0},d_{1}+1} & \alpha_{d_{0},d_{1}+2} & \cdots & \alpha_{d_{0},d_{1}+k}
\\\hline \end{array}\\
\begin{array}{|ccc|}
\hline \alpha_{d_{0}+1,1} & \cdots & \alpha_{d_{0}+1,d_{1}}\\
\alpha_{d_{0}+2,d_{1}} & \cdots & \alpha_{d_{0}+2,d_{1}}\\
\alpha_{d_{0}+3,d_{1}} & \cdots & \alpha_{d_{0}+3,d_{1}}\\
\vdots &  & \vdots\\
\alpha_{d_{0}+k,d_{1}} & \cdots & \alpha_{d_{0}+k,d_{1}}
\\\hline \end{array}
\end{array}
\]

\[
\begin{array}{cc}
 & \begin{array}{|cccc|}
\hline x_{1,d_{1}+1} & x_{1,d_{1}+2} & \cdots & x_{1,d_{1}+k}\\
x_{2,d_{1}+1} & x_{2,d_{1}+2} & \cdots & x_{2,d_{1}+k}\\
\vdots & \vdots &  & \vdots\\
x_{d_{0},d_{1}+1} & x_{d_{0},d_{1}+2} & \cdots & x_{d_{0},d_{1}+k}
\\\hline \end{array}\\
\begin{array}{|ccc|}
\hline x_{d_{0}+1,1} & \cdots & x_{d_{0}+1,d_{1}}\\
x_{d_{0}+2,d_{1}} & \cdots & x_{d_{0}+2,d_{1}}\\
x_{d_{0}+3,d_{1}} & \cdots & x_{d_{0}+3,d_{1}}\\
\vdots &  & \vdots\\
x_{d_{0}+k,d_{1}} & \cdots & x_{d_{0}+k,d_{1}}
\\\hline \end{array}
\end{array}
\]

\bigskip

We construct a \textit{long} monomial consisting of elements of $A$ as follows.

For each $\mathbb{Z}_{2}$-graded simple component we write a nonzero product of the standard basis, namely elements of the form $e_{i,j} \in M_{l,f}(F)$ or $u_{g} \otimes e_{i,j} \in FC_{2} \otimes M_{n}$ where $g = e, \sigma$. It is known that such a product exists. We refer to these elements as \textit{designated} elements. In order to keep a unified notation we shall replace $e_{i,j} \in M_{l,f}(F)$ by $u_{e} \otimes e_{i,j}$. Furthermore, we may assume for simplicity that the nonzero product starts (resp. ends) with an element of the form $u_{e}\otimes e_{1,y}$ (resp. $u_{g}\otimes e_{x,1}$). Next we border each basis element $u_{g}\otimes e_{i,j}$ from left (resp. right) with the element $u_{e}\otimes e_{i,i}$ (resp. $u_{e} \otimes e_{j,j}$) which we call \textit{frame}, so that the product of the monomial remains nonzero. Let us denote the product above, namely the product corresponding to the $\mathbb{Z}_{2}$-graded simple algebra $A_{i}$ by $Z_{i}$. We take now the product of $k$ copies of this monomial $Z_{i,1}\cdots Z_{i,k}$. This is clearly nonzero. Next, we bridge the $\mathbb{Z}_{2}$-graded simple components with appropriate radical values $w_{s, s+1}$ and get a nonzero product as dictated by the expression $A_{1}JA_{2}\cdots JA_{q} \neq 0$.

Finally, we tensor the basis elements with Grassmann elements, where even elements of $A$ are tensored with $1$ and odd elements are tensored with \textit{different} generators $\epsilon_{i}$ (odd order). We shall always view these tensors as ungraded elements of $E(A)$ although, abusing language, we will refer to them as even and odd elements respectively.

We obtained a nonzero expression of the form

$$Z_{1,1}\cdots Z_{1,k}\cdot w_{1,2}\cdot Z_{2,1}\cdots Z_{2,k}\cdot w_{2,3} \cdots w_{q-1,q}\cdot Z_{q,1}\cdots Z_{q,k}.$$
Consider the set $U_{even, 1}$ of designated even elements in the tuple $$(Z_{1,1}, Z_{2,1}, \ldots, Z_{q,1}).$$ Similarly we let $U_{even, i}$, $i=1,\ldots,k$ be the designated even elements in the tuple $(Z_{1,i}, Z_{2,i}, \ldots, Z_{q,i})$. Observe that the cardinality of $U_{even, i}$ is $d_{o} = dim_{F}A_{ss, 0}$.

We denote the elements of $U_{even, i}$ by $\alpha_{1, d_{1}+i},\ldots, \alpha_{d_{0}, d_{1}+i}$, that is as the $i$th column of the horizontal strip above. Furthermore, it will be convenient to denote the elements $(Z_{1,1}, Z_{2,1}, \ldots, Z_{q,1})$ in the same order as they appear in the $i$th column.

Similarly, $U_{odd, j}$ consists of all designated odd elements in the tuple
$$(Z_{1,i}, Z_{2,i}, \ldots, Z_{q,i})$$
and we denote them using the notation of the elements in the $j$th row of the vertical strip.

For each $t = 1,\ldots,k$ we alternate the designated (even) elements $$\alpha_{1, d_{1}+t},\ldots, \alpha_{d_{0}, d_{1}+t}$$ and symmetrize the designated (odd) elements $\alpha_{d_{0}+t, 1},\ldots, \alpha_{d_{0}+t, d_{1}}$. We claim the expression obtained is nonzero. Indeed, any \textit{nontrivial} permutation (independently of its sign) of designated even elements will be surrounded by frames where not all match and hence will vanish. Similarly with the odd elements of $A$. In particular alternating the even elements and symmetrizing the odd elements yields a nonzero value.

We now \textit{symmetrize} the sets of $k$ elements corresponding to the rows of the horizontal strip and \textit{alternate} the sets of $k$ elements corresponding to the columns of the vertical strip. We claim we get a nonzero value. For the proof we may assume each tuple of $k$ even elements are equal and are of the form $u_{e}\otimes e_{i,j}$ whereas for the odd elements we assume as we may, the elements of each $k$ tuple have the form $\epsilon_{i,j,g}\otimes u_{g}\otimes e_{i,j}$, $g = \{e, \sigma \}$, $\epsilon_{i,j,g}$ are generators of the Grassmann algebra and the elements $u_{g}\otimes e_{i,j}$ of $A$ are equal. It follows that symmetrization of the rows in the horizontal strip and alternation of the columns in the vertical strip yield the multiplication of each monomial by a factor of $(k!)^{d_{0}}$. In particular if the respective operations were performed on a vanishing product it remains zero whereas it is nonzero if the operations were performed on a nonvanishing product.

We now replace the elements of $E(A)$ appearing in the monomial
$$Z_{1,1}\cdots Z_{1,k}\cdot w_{1,2}\cdot Z_{2,1}\cdots Z_{2,k}\cdot w_{2,3} \cdots w_{q-1,q}\cdot Z_{q,1}\cdots Z_{q,k}$$ by variables which we call designated variables, frames and bridges. Note that the monomial obtained is in $\langle f_{0} \rangle_{T}$ where $f_{0} = X_{1}\cdot w_{1} \cdots w_{q-1}\cdot X_{q}$.  It is convenient to arrange the designated variables $x_{r,s}$ in two strips in $1-1$ correspondence with the designated elements $\alpha_{r,s} \in E(A)$. Finally, we perform the alternations and symmetrizations on the variables and obtain (by construction) a multilinear nonidentity of $E(A)$.

We denote the polynomial by $p$. We summarize the paragraph above in the following proposition.

\begin{proposition}
Let $A$ be a finite dimensional $\mathbb{Z}_{2}$-graded algebra over $F$. Suppose $A$ is full. Let $p=p_{A}$ be the polynomial constructed above. Then $p$ is a nonidentity of $E(A)$. Furthermore, $p \in \langle f_{0} \rangle_{T}$ where $f_{0} = X_{1}\cdot w_{1} \cdots w_{q-1}\cdot X_{q}$.
\end{proposition}

\begin{proposition}\label{k large enough}
For $k$ large enough, the polynomial $p$ is strongly full of $E(A)$.

\end{proposition}

\begin{proof}
Suppose this is not the case. We claim that only in a bounded number of columns in the horizontal strip of the diagram we can put either radical elements or odd semisimple elements. Indeed, it is clear that the number of radical values is bounded. If we put arbitrary many odd semisimple values, by the pigeonhole principle, there will be variables in the same row which will get values of the form $\epsilon_{i_1}\cdots\epsilon_{i_n} \otimes a$ and $\epsilon_{j_1}\cdots\epsilon_{j_m}\otimes a$, same $a$, where $n$ and $m$ are odd. Then the symmetrization of the corresponding variables yields zero. Similarly, in any nonzero evaluation, the number of rows in the vertical strip of the diagram in which we can put radical or even elements is bounded. It follows then that for large enough $k$ there exists a column, say the $i$th column, in the horizontal strip which assumes only even elements and there is a $j$th row in the vertical strip which assumes only odd elements. But more than that, taking $k$ large enough we may assume $i = j$. It follows that by the alternation of the columns in the horizontal strip (resp. symmetrization of the rows in the vertical strip), in any nonzero evaluation, we are forced to evaluate these on basis elements of the form $\epsilon_{i_1}\cdots\epsilon_{i_n} \otimes a$ where $a$ runs over a full basis of $A_{ss,0}$ (resp. $A_{ss,1}$). This proves the first part of the lemma. For the second part of the lemma, we start with an arbitrary multilinear weakly full of $E(A)$ and we fix a nonzero admissible evaluation $\Phi$ of $f_{0}$ which visits all $\mathbb{Z}_{2}$-graded simple components of $A_{ss}$. Denote by $X_{1},\ldots,X_{q}$ the variables of $f_{0}$ which assume values from the $q$ different. Applying the $T$-operation we replace the variables $X_{1}, \ldots ,X_{q}$ with $X_{1}\Delta_{1}, \ldots, X_{q}\Delta_{q}$ where $\Delta_{t} = $$Z_{t,1}\cdots Z_{t,k}$. Finally we alternate and symmetrize the designated variables as above. The polynomial obtained $p \in \langle f_{0} \rangle_{T}$ is strongly full for the algebra $E(A)$. The proof is similar to the proof above when $f_{0}$ is a monomial. Details are omitted. This completes the proof of Proposition \ref{k large enough} and also of the Main Lemma.

\end{proof}

\begin{theorem}
Let $B$ be a finite dimensional super algebra and suppose it does not cover the super algebra $A$. Let $p = p_{A}$ be the polynomial constructed above. Then $p$ is an identity of $E(B)$.
\end{theorem}
\begin{proof}
The proof is similar to the proof in the affine case. In any nonzero evaluation on $E(B)$ we must have an index $i$ which obtains linearly independent semisimple elements of $B$. If the evaluation is nonzero, we must have a monomial with nonzero value and hence the semisimple elements appearing in each segment must come from the same $\mathbb{Z}_{2}$-graded simple component of $B$. We have then that $B$ covers $A$. Contradiction.
\end{proof}

\begin{corollary}
Let $A$ and $B$ full super algebras. If $E(A)$ and $E(B)$ are PI-equivalent, then their semisimple parts, $A_{ss}$ and $B_{ss}$ are isomorphic.
\end{corollary}
\begin{proof}
Indeed, $B$ and $A$ cover each other. It follows that the tuple of pairs of \textit{dimensions} of the simple components of $A$ and $B$ coincide (up to a permutation). Finally we note (see below) that the super structure of a super simple algebra $A$ is determined by the dimensions of $A_{0}$ and $A_{1}$ and hence if these coincide, $A_{ss}$ and $B_{ss}$ must be isomorphic as super algebras.
\end{proof}
For the rest of the proof, we follow step by step the proof in the affine case. Along the proof two basic propositions are needed.

\begin{proposition}
Let $A$ be a finite dimensional super algebra over $F$. Then $E(A)$ is PI-equivalent to the direct sum of algebras $E(A_{i})$ where $A_{i}$ is a finite dimensional full super algebra.
\end{proposition}
\begin{proof}
Recall that a finite dimensional super algebra $A$ is PI-equivalent to the direct sum of full super algebras $\mathfrak{A} = A_{1}\oplus \cdots \oplus A_{n}$. We claim firstly: $E(A)$ and $E(\mathfrak{A})$ are PI-equivalent: Indeed, a super polynomial $f$ is an identity of $A$ if and only if the super polynomial $f^{*}$ is a super identity of $E(A)$ as a super algebra where the $0$ component is spanned by elements of the form $\epsilon_{i_1}\cdots \epsilon_{i_2r} \otimes a_{0}$ and the $1$-component is spanned by elements of the form $\epsilon_{i_1}\cdots \epsilon_{i_{2r+1}} \otimes a_{1}$. Here $\epsilon_{j}$ is a generator of $E$, $a_{0} \in A^{0}$,$a_{1} \in A^{1}$, the even and odd elements of $A$ respectively. Then, if $E(A)$ and $E(\mathfrak{A})$ are PI-equivalent as super algebras, they are PI-equivalent as ungraded algebras. Next we argue that $E(\mathfrak{A}) \cong E(A_{1}) \oplus \cdots \oplus E(A_{n})$ and the proposition is proved.
\end{proof}

The second statement we need is

\begin{proposition}
Let $A$ and $B$ be finite dimensional super simple algebras over $F$. If $dim_{F}(A^{0}) = dim_{F}(B^{0})$ and $dim_{F}(A^{1}) = dim_{F}(B^{1})$ then $A$ and $B$ are $\mathbb{Z}_{2}$-graded isomorphic.
\end{proposition}
\begin{proof}
Recall that a $\mathbb{Z}_{2}$-graded simple algebra over an algebraically closed field $F$ of characteristic $0$ is isomorphic to $M_{l,f}(F)$, where $l \geq 1, f \geq 0$ or $FC_{2} \otimes_{F} M_{n}(F)$, $n\geq 1$. In the case of $M_{l,f}(F)$ the dimension of the $0$-component (resp. $1$-component) is $l^{2} + f^{2}$ (resp. $2lf$) and in particular the total dimension is an integer square whereas in the case of  $FC_{2} \otimes_{F} M_{n}(F)$ the dimensions of the homogeneous components are each equal to $n^{2}$ and hence not an integer square. This proves the proposition.
\end{proof}

\section{$G$-graded algebras}\label{Section G-graded algebras}
In this section we extend the main theorem to the setting of $G$-graded algebras where $G$ is a finite group. Here is the precise statement.

\begin{theorem}\label{uniquely determined graded affine}
Let $G$ be a finite group and let $\Gamma$ be a $G$-graded $T$-ideal over $F$. Suppose $\Gamma$ contains an ungraded Capelli polynomial $c_{n}$, some $n$. Then there exists a finite dimensional semisimple $G$-graded algebra $U$ over $F$ which satisfies the following conditions.

\begin{enumerate}

\item
There exists a finite dimensional $G$-graded algebra $A$ over $F$ with $Id_{G}(A) = \Gamma$ and such that $A \cong U \oplus J_{A}$ is its Wedderburn-Malcev decomposition as $G$-graded algebras.

\item

If $B$ is any finite dimensional $G$-graded algebra over $F$ with $Id_{G}(B) = \Gamma$ and $B_{ss}$ is its maximal semisimple $G$-graded subalgebra, then $U$ is a direct summand of $B_{ss}$ as $G$-graded algebras.

\end{enumerate}
\end{theorem}

The proof basically follows the main lines of the proof of the ungraded case but yet there is a substantial obstacle due to the fact that $G$-graded simple algebras are \textit{not determined} up to isomorphism by the dimensions of the corresponding homogeneous components. In the following examples, as usual, $F$ is an algebraically closed field of characteristic zero.
\begin{example}

1. If $G$ is a finite group, $F^{\alpha}G$ and $F^{\beta}G$, $\alpha, \beta \in H^{2}(G, F^{*})$, are twisted group algebras, then they are $G$-graded isomorphic if and only if $\alpha = \beta$. Clearly, the dimensions of the homogeneous components equal $1$ independently of the cohomology class.

2. Let $G = \{e, \sigma, \tau, \sigma\tau\}$ be the Klein $4$-group. We  consider the crossed product grading on $A \cong M_{4}(F)$, that is the elementary grading determined by the tuple $(e, \sigma, \tau, \sigma\tau)$ and also the algebras $B_{i} \cong F^{\beta_{i}}G \otimes M_{2}(F)$, $\beta_{1}, \beta_{2}\in H^{2}(G,F^{*})$. Here $\beta_{1}$ (resp. $\beta_{2})$ is the trivial (resp. nontrivial) cohomology class on $G$ with values on $F^{*}$.  The dimension of each homogeneous component is $4$. It is easy to show the algebras $A, B_{1}, B_{2}$ are nonisomorphic as $G$-graded algebras (see \cite{AljHaile}). Note however that $A \cong B_{2} \cong M_{4}(F)$ and $B_{1} \cong M_{2}(F) \oplus M_{2}(F) \oplus M_{2}(F) \oplus M_{2}(F)$ as ungraded algebras.

\end{example}

Let $G$ be a finite group and let $A$ be a finite dimensional $G$-graded algebra over $F$. We decompose $A$ into $A_{ss} \oplus J$ where $A_{ss}$ is a maximal $G$-graded semisimple algebra which supplements $J$, the Jacobson radical. The algebra $A_{ss}$ decomposes into a direct product of $G$-graded simple components $A_{1} \times \cdots \times A_{q}$. As in the ungraded case, the $G$-graded simple components are uniquely determined up to a $G$-graded isomorphism.

We start with the definition of the covering relation.

\begin{definition}
Let $Q$ and $V$ be finite dimensional $G$-graded semisimple algebras over $F$. We say $V$ covers $Q$ if the $G$-graded simple components of $Q$ can be decomposed into subsets such that the sum of the dimensions of the corresponding homogeneous components are bounded by the dimensions of the homogeneous components of $V$. More precisely, if $Q \cong Q_{1}\times \cdots \times Q_{q}$ and $V \cong V_{1} \times \cdots \times V_{r}$ are the decompositions of $Q$ and $V$ into their $G$-graded simple components. Let $u_{i,g} = dim_{F}(Q_{i})_{g}$ (resp. $v_{j,g} = dim_{F}(V_{j})_{g}$) denote the dimension of the $g$-homogeneous component of $Q_{i}$ (resp. of $V_{j}$). Then $V$ covers $Q$ if and only if the indices $1,\ldots,q$ can be decomposed into $r$ subsets $\Lambda_{1}, \ldots, \Lambda_{r}$ such that $\sum_{i \in \Lambda_{j}}u_{i,g} \leq v_{j,g}$.

\end{definition}

\begin{definition}
Let $A$ be a finite dimensional $G$-graded algebra over $F$. Let $A \cong A_{1} \times \cdots \times A_{q} \oplus J_{A}$ be its Wedderburn-Malcev decomposition, where $A_{i}$ are $G$-graded simple. We say $A$ is full if up to a permutation of the indices we have $A_{1}\cdot J\cdots J\cdot A_{q} \neq 0$.
\end{definition}

Next we introduce $G$-graded strongly full polynomials for a given full algebra $A$. These are multilinear nonidentities $f$ of $A$ which vanish when evaluated on $A$ unless every basis element of $A_{ss}$ appears as a value of one of its variables. These polynomials were constructed in \cite{AB}. Nevertheless, we shall need their precise structure so let us recall here their construction.

For each $G$-graded simple component $A_{i}$ of $A$ consider a \textit{nonzero} product of all basis elements of $A_{i}$. These are elements of the form $u_{h}\otimes e_{i,j}$, where the homogeneous degree is $g_{r}^{-1}hg_{s}$. Here, the grading on $A_{i}$ is determined by a presentation $P_{A_i}$ (see \cite{BSZ} and for the precise notation see \cite{AljHaile} Theorem $1.1$). It is known that such a product exists (see \cite{AB}). As above, we border from right and left each basis element with frames of the form $u_{e} \otimes e_{i,i}$. We denote such product of basis elements, namely of designated and frame elements, by $Z_{i}$. We refer to $Z_{i}$ as the monomial of basis elements of $A_{i}$. We may assume the product starts with an element of the form $u_{e} \otimes e_{1,1}$ and ends with an element of the form $u_{h}\otimes e_{r,1}$ and so if $Z_{i,l} = Z_{i}$, $l=1,\ldots,k$, we have that the product $Z_{i,1}\cdots Z_{i, k}$ is nonzero. Next we bridge products corresponding to different $G$-graded simple components by radical (homogeneous) elements $w_{i}$. We obtain a nonzero product

 $$Z_{1,1}\cdots Z_{1, k}\cdot w_{1} Z_{2,1}\cdots Z_{2, k}\cdot w_{2}\cdots w_{q-1}\cdot Z_{q,1}\cdots Z_{q, k}.$$
As in the ungraded case we consider the $i$th set $\Lambda_{i}$, $i =1,\ldots,k$ consisting of the designated (semisimple) elements in $Z_{1,i},\ldots,Z_{q,i}$. We denote by $\Lambda_{i, g}$, $g\in G$, the subset of $\Lambda_{i}$ consisting of elements of homogeneous degree $g$. We claim any nontrivial permutation of designated elements in $\Lambda_{i, g}$ yields a zero product. To see this we note that if $u_{h_1}\otimes e_{r_1,s_1} \neq u_{h_2}\otimes e_{r_2,s_2}$ are basis elements of \textit{equal} homogeneous degree, that is $g_{r_1}^{-1}h_1 g_{s_1} = g_{r_2}^{-1}h_2 g_{s_2}$, we must have $(r_{1}, s_{1}) \neq (r_{2}, s_{2})$. This implies that frames bordering different designated elements of the same homogeneous degree are different (also for those elements that belong to the same $G$-graded simple component). This proves the claim.

We proceed as in the ungraded case where the monomials consisting of elements of $A$ are replaced by monomials of different graded variables with the corresponding homogeneous degree. The \textit{small} sets here are alternating sets of variables of degree $g \in G$ of cardinality equal the dimension of the $g$-homogeneous component of $A_{ss}$. The polynomial obtained is denoted by $p$. This completes the construction of a $G$-graded strongly full polynomial of $A$. As in previous cases we shall need a more general statement.

\begin{definition}
Let $A$ be a finite dimensional full $G$-graded algebra. A multilinear $G$-graded polynomial is weakly full if it admits an admissible nonzero evaluation on $A$ which visits each $G$-graded simple subalgebra of $A_{ss}$. The definitions of a \textit{full} and \textit{strongly full} $G$-graded polynomial are similar to the definitions in the affine case.
\end{definition}
\begin{proposition}
Let $A$ be a full $G$-graded algebra and $f_{0}$ a $G$-graded multilinear polynomial which is weakly full of $A$. Then there exists a multilinear $G$-graded strongly full polynomial $f$ such that $f \in \langle f_{0} \rangle_{T}$.
\end{proposition}
\begin{proof}
The proof is similar to the proof of Theorem \ref{full polynomials existence}, part $3$.
\end{proof}
\begin{lemma}
Let $A$ be a $G$-graded full algebra and $p$ a $G$-graded strongly full polynomial on $A$ with sufficiently many small sets. If $B$ does not cover $A$, then $p$ is an identity of $B$.
\end{lemma}
\begin{proof}
The proof is similar to the proof of Lemma \ref{principal lemma}.
\end{proof}
Note that in the ungraded case this was sufficient for deducing that the semisimple subalgebras of $A$ and $B$ are isomorphic.
\begin{theorem}\label{isomorphic graded semisimple algebras}
Let $A$ and $B$ be finite dimensional $G$-graded full algebras. Suppose they are $G$-graded PI-equivalent. Then the maximal semisimple subalgebras $A_{ss}$ and $B_{ss}$ are $G$-graded isomorphic.
\end{theorem}
\begin{proof}

By the preceding lemma we know that $A$ and $B$ cover each other and hence the tuples of the dimensions of the homogeneous components of the $G$-graded simple algebras appearing in the decomposition of $A_{ss}$ and $B_{ss}$ are equal. Our goal is to show the corresponding $G$-graded simple components are $G$-graded isomorphic.

For the proof we shall need to insert \textit{suitable} $e$-central polynomial in the full $G$-graded polynomials of $A$ constructed above.
We recall from \cite{KarasikGcentral} that every finite dimensional $G$-graded simple admits an $e$-central multilinear polynomial $c_{A}$, that is a nonidentity of $A$, central and $G$-homogeneous of degree $e$.  Furthermore, it follows from its construction, that the polynomial $c_{A}$ alternates on certain sets of variables of equal homogeneous degree of cardinality equal $dim_{F}(A_{g})$, for every $g \in G$. For the proof of Theorem \ref{isomorphic graded semisimple algebras} we shall need $e$-central polynomials with some additional properties.

\begin{theorem}
Let $A_{i}$, $i=1,\ldots,q$, be the simple components of $A_{ss}$. Then there exists a polynomial $m_{i}(X_{G})$ with the following properties.

\begin{enumerate}
\item

$m_{i}(X_{G})$ is $e$-central of $A_{i}$.

 \item

$m_{i}(X_{G})$ is an identity of every algebra $\Sigma$ which satisfies the following conditions.

\begin{enumerate}
\item
$\Sigma$ is finite dimensional $G$-graded simple.
\item
$dim_{F}(\Sigma_{g}) = dim_{F}((A_{i})_{g})$ for every $g \in G$.
\item
$Id_{G}(A_{i}) \nsupseteq Id_{G}(\Sigma)$.

\end{enumerate}
\end{enumerate}

\end{theorem}

\begin{proof}
By our assumption $(2c)$, there is a $G$-graded homogeneous nonidentity $f_{i,\Sigma}$ of $A_{i}$, of homogeneous degree $g \in G$ say, which vanishes on $\Sigma$ and so replacing a variable of degree $g$ in an alternating set of $c_{A_{i}}$ by $f_{i,\Sigma}$ yields a nonidentity $e$-central polynomial $m_{i, \Sigma}(X_{G})$ of $A_{i}$ which vanishes on $\Sigma$. Now, we recall from \cite{AljKarsikAzumaya} that the number of $G$-graded simple algebras $\Sigma$ satisfying conditions $(2a,2b)$ above is finite and so, because the nonzero values of $m_{i,\Sigma}(X_{G})$ are invertible $(in F^{*})$, we have that $m_{i}(X_{G}) = \Pi_{\Sigma}m_{i,\Sigma}(X_{G})$ is an $e$-central polynomial of $A$ with the desired properties.
\end{proof}

Finally we insert a polynomial with disjoint sets of variables $m_{i}(X_{G})$ adjacent to each monomial $Z_{i,l}$. This completes the construction of the strongly full polynomial $p$.

We can complete now the proof of the Theorem \ref{isomorphic graded semisimple algebras}. We are assuming the algebras $A$ and $B$ are PI-equivalent and so by Theorem 4.4, the algebras $A$ and $B$ cover each other.
It follows that $A_{ss}$ and $B_{ss}$ have the same number of $G$-graded simple components. Furthermore, if $A_{ss} \cong A_{1} \times \cdots \times A_{q}$ and $B_{ss} \cong B_{1} \times \cdots \times B_{q}$ then there is a permutation $\sigma \in Sym(q)$ such that

\begin{enumerate}
\item
$dim_{F}((A_{i})_g) = dim_{F}((B_{\sigma(i)})_g)$, $i=1,\ldots,q$ and every $g \in G$.

\end{enumerate}

We claim there is a permutation of the $G$-graded simple components of $B_{ss}$ such that in addition to the condition above we have that
$Id_{G}(A_{i}) \supseteq Id_{G}(B_{\sigma(i)})$, $i = 1,\ldots, q$. Suppose not. Then for every permutation $\sigma$ satisfying the condition above there is a $j = j(\sigma)$ such that $Id_{G}(A_{j}) \nsupseteq Id_{G}(B_{\sigma(j)})$. We will show that the strongly full polynomial $p$ is an identity of $B$, in contradiction to the PI-equivalence of $A$ and $B$. Indeed, evaluating $p$ on $B$, the value will be zero unless there is a monomial $Z_{i}$, together with the inserted central polynomials, whose value is nonzero. This implies there is a permutation $\sigma$ of the components of $B_{ss}$ such that the $i$th segment of $p$ is evaluated on $B_{\sigma(i)}$. This already implies the condition above on the dimensions. But by assumption there is $j$ such that $Id_{G}(A_{j}) \nsupseteq Id_{G}(B_{\sigma(j)})$ and so the central polynomial $m_{j}(X_{G})$ vanishes on $B_{\sigma(j)}$.

We conclude there is a permutation $\sigma \in Sym(q)$ of the simple components of $B_{ss}$ such that
\begin{enumerate}
\item
$dim_{F}((A_{i})_g) = dim_{F}((B_{\sigma(i)})_g)$, $i=1,\ldots,q$, and every $g \in G$

\item

$Id_{G}(A_{i}) \supseteq Id_{G}(B_{\sigma(i)})$, $i = 1,\ldots, q$.
\end{enumerate}

Our goal is to show that in fact $Id_{G}(A_{i}) = Id_{G}(B_{\sigma(i)})$, $i = 1,\ldots, q$. Indeed, this implies what we need, that is $A_{i} \cong  B_{\sigma(i)}$, $i = 1,\ldots, q$, as $G$-graded algebras (see \cite{AljHaile}).

Suppose that $G$ is abelian. In that case let us recall the following general result of O. David (see \cite{DavidOfir}).

\begin{theorem} Let $G$ be a finite abelian group and let $A$ and $B$ finite dimensional $G$-graded simple algebras over an algebraically closed field $F$. Then $A \hookrightarrow B$ as $G$-graded algebras if and only if $Id_{G}(A) \supseteq Id_{G}(B)$.
\end{theorem}
Clearly, it follows at once from the theorem that $G$-graded algebras satisfying conditions $(1)$ and $(2)$ above must be $G$-graded isomorphic. David's result is not known in case $G$ is an arbitrary finite group.

Here instead, we argue as follows. By symmetry there is a permutation $\tau \in Sym(q)$ such that
\begin{enumerate}
\item
$dim_{F}((B_{i})_g) = dim_{F}((A_{\tau(i)})_g)$, $i=1,\ldots,q$ and every $g \in G$

\item

$Id_{G}(B_{i}) \supseteq Id_{G}(A_{\tau(i)})$, $i = 1,\ldots, q$
\end{enumerate}

Consequently there is a permutation $\rho \in Sym(q)$ such that $A_{i}$ and $A_{\rho(i)}$ have equal dimensions of the homogeneous components and $Id_{G}(A_{i}) \supseteq Id_{G}(A_{\rho(i)})$. We need to show equality holds. Indeed, we see that $Id_{G}(A_{i}) = Id_{G}(A_{j})$ for $i$ and $j$ which belong to the same orbit determined by $\rho$ and so, in particular $Id_{G}(A_{i}) = Id_{G}(A_{\rho(i)})$, $i=1,\ldots, q$.

\end{proof}

The remaining steps in the proof of Theorem \ref{uniquely determined graded affine} are similar to those in the proof of Theorem \ref{uniquely determined affine}. Details are omitted.

\section{PI-equivalence of Grassmann envelopes of finite dimensional $G_{2}$-graded algebras}

In this section we treat the case where the algebra $A$ is finite dimensional $\mathbb{Z}_{2}\times G$-graded and $E(A)$ is the Grassmann envelope of $A$ viewed as a $G$-graded algebra.

The main result in this case is the following.

\begin{theorem}\label{uniquely determined graded nonaffine}
Let $G$ be a finite group. Let $\Gamma$ be a $G$-graded $T$-ideal. Suppose $\Gamma$ contains a nonzero ungraded polynomial but contains no ungraded Capelli $c_{n}$, any $n$. Then there exists a finite dimensional semisimple $\mathbb{Z}_{2} \times G$-graded algebra $U$ over $F$ which satisfies the following conditions.

\begin{enumerate}

\item

There exists a finite dimensional $\mathbb{Z}_{2} \times G$-graded algebra $A$ over $F$ with $Id_{G}(E(A)) = \Gamma$ and such that $A \cong U \oplus J_{A}$ is its Wedderburn-Malcev decomposition as $\mathbb{Z}_{2} \times G$-graded algebras.

\item

If $B$ is any finite dimensional $\mathbb{Z}_{2} \times G$-graded algebra algebra over $F$ with $Id_{G}(E(B)) = \Gamma$ and $B_{ss}$ is its maximal semisimple $\mathbb{Z}_{2} \times G$-graded subalgebra, then $U$ is a direct summand of $B_{ss}$ as $\mathbb{Z}_{2} \times G$-graded algebras.

\end{enumerate}

\end{theorem}

The general approach is based on cases that were treated earlier, namely the cases where ($1$) $\Gamma$ is a $T$-ideal of identities of a $G$-graded affine algebra ($2$) $\Gamma$ is a $T$-ideal of identities of an ungraded nonaffine algebra. It turns out however, that also here there is a substantial difficulty, and this is in the very first step of the general approach (see Theorem \ref{theorem:different G2 simples different envelopes} below). In fact, nearly the entire section will be devoted to the proof of Theorem \ref{theorem:different G2 simples different envelopes}.

Before we state the theorem let us set some notation.

Let $G$ be a finite group and denote $G_{2}:=\mathbb{Z}_{2}\times G$.
We denote $G_{even}:=0\times G;G_{odd}=1\times G$ and similarly for
a $G_{2}$ algebra $A$ we write $A_{even}=A_{G_{even}};A_{odd}=A_{G_{odd}}$.

\begin{theorem}
\label{theorem:different G2 simples different envelopes}Suppose that
$A$ and $B$ are two finite dimensional $G_{2}$-graded simple algebras.
Then, $A$ and $B$ are $G_{2}$-graded isomorphic if and only if $E(A)$ and $E(B)$
have the same $G$-graded identities.
\end{theorem}

It is worth noting that the Grassmann $*$ operation allows one to
pass from a super identity of $A$ to a super identity of $E(A)$ (resp. from a $G_{2}$-identity of $A$ to a $G_{2}$-identity of $E(A)$). The challenge here lies in transforming a super identity of $A$ into an ordinary identity of $E(A)$ (resp. from a $G_{2}$-identity of $A$ into a $G$-identity of $E(A)$).

The main part of the proof
of the above Theorem is to find such a transformation. We start this
section with the construction of the transformation and in Proposition
\ref{proposition:forcing G evaluation} we show the key property that makes
it work. We emphasize that the construction and also the Theorem are
guaranteed to work only in the case where the algebras in question
are finite dimensional $G_{2}$-graded simple. In the general case it is not true that if $E(A)$
and $E(B)$ have the same $G$-graded identities then $A$ and $B$ have
the same $G_{2}$ identities. An example can be found in (\cite{GiamZaicevBook}, Section $8.2$).

The construction we are about to show is a generalization of the construction
in Section \ref{nonaffine algebras section}  and in fact, the main property of this construction appearing
in Proposition \ref{proposition:forcing G evaluation} is clearly true for
the previous construction. However, with the previous construction we can
only show that if $E(A)$ and $E(B)$ have the same $G$-identites
then $\dim A_{\bar{g}}=\dim B_{\bar{g}}$ for all $\bar{g}$. As pointed out earlier, for
general groups $G$ one can easily find examples of nonisomorphic
$G_{2}$-graded simple algebras having this property.

Let $f=f(X_{0};Y_{0})$ be a multilinear $G_{2}$-graded polynomial,
where
$$X_{0}=\coprod_{\bar{g}\in G_{2}}\coprod_{i=1}^{T}X_{\bar{g},i}$$
is a union of $T$ \textit{small sets} of degree $\bar{g}$-variables
$X_{\bar{g},i}=\{x_{\bar{g},i}^{(1)},\dots,x_{\bar{g},i}^{(\dim A_{\bar{g}})}\}$
(here $\bar{g}$ runs over all of $G_{2}$), and $Y_{0}=\coprod_{\bar{g}\in G_{2}}Y_{\bar{g},0}$
are some additional variables. Assume that $f$ has a $G_{2}$-graded
evaluation $\phi:F\left\langle X_{0};Y_{0}\right\rangle \to A$ with
the following properties:
\begin{enumerate}
\item For every nontrivial permutation $\sigma\in\prod_{\bar{g}\in G_{2}}\prod_{i=1}^{T}S_{X_{\bar{g},i}}$
(here $S_{W}$ is the symmetric group on the set $W$) the value of
$f(\sigma(X_{0});Y_{0})$ under the evaluation $\phi$ is $0$.
\item For all $\bar{g}\in G_{2}$ the value $\phi(x_{\bar{g},i}^{(j)})=:a_{\bar{g}}^{(j)}$
is independent of $i=1,\dots,T$. Furthermore, all $a_{\bar{g}}^{(j)}$
are distinct.
\end{enumerate}
We will see later that in the case relevant to the proof of Theorem
\ref{theorem:different G2 simples different envelopes} it is indeed possible
to construct such a polynomial.

Let $k>0$ be a natural number and consider the polynomial
\[
f_{k}:=f(X_{1};Y_{1})\cdots f(X_{k};Y_{k}),
\]
where all $X_{t}$ and $Y_{t}$ are disjoint copies of $X_{0}$ and
$Y_{0}$ respectively. Notice that $X_{t}=\coprod_{\bar{g}\in G_{2}}\coprod_{i=1}^{T}X_{\bar{g},(t-1)T+i}$.
We extend $\phi$ to $F\left\langle X;Y\right\rangle $, where $X=\coprod_{t=1}^{k}X_{t}$
and $Y=\coprod_{t=1}^{k}Y_{t}$, by duplicating the evaluation on
$X_{0}$ and $Y_{0}$ to $X_{t}$ and $Y_{t}$ respectively (for all
$t=1,\dots,k$). As a result, we have in particular for all $\bar{g},i$
and $j$ that $\phi(x_{\bar{g},i}^{(j)})=a_{\bar{g}}^{(j)}$ (we rely
here on property (2)).

For $a\in A$ we set $X_{\phi}(a)\subset X$ to be all the variables
from $X$ which $\phi$ assigns to them the value $a$. In other words,
$X_{\phi}(a)=(\phi|_{X})^{-1}(a)$. In particular, $X_{\phi}(a_{\bar{g}}^{(j)})=\{x_{\bar{g},1}^{(j)},\dots,x_{\bar{g},kT}^{(j)}\}$.
\begin{remark}
\label{remark:table structure}For every $\bar{g}_{0}\in G_{2}$ we have
\[
\coprod_{i=1}^{kT}X_{\bar{g}_{0},i}=\coprod_{j=1}^{\dim A_{\bar{g}}}X_{\phi}(a_{\bar{g}_{0}}^{(j)}).
\]
One should visualize this equality as ``union of columns'' (the
$X_{\bar{g}_{0},i}$'s) = ``union of rows'' (the $X_{\phi}(a_{\bar{g}_{0}}^{(j)})$)
in the matrix
\[
\begin{bmatrix}x_{\bar{g},1}^{(1)} &  & \cdots &  & x_{\bar{g},kT}^{(1)}\\
\\
\vdots &  & x_{\bar{g},i}^{(j)} &  & \vdots\\
\\
x_{\bar{g},1}^{(\dim A_{\bar{g}})} &  & \cdots &  & x_{\bar{g},kT}^{(\dim A_{\bar{g}})}
\end{bmatrix}.
\]
\end{remark}

Next, we alternate and symmetrize different subsets of $X$ in the
following fashion to obtain a new graded polynomial $s_{k;A}$. For
each even (odd) element $\bar{g}\in G_{2}$ we apply alternation (symmetrization)
on all variables $X_{\bar{g},i}$; afterwards we apply symmetrization
(alternation) for every set of variables of the form $X_{\phi}(a)$.
All in all, we have

\bigskip
$s_{k;\phi;A}(f)  =$

$$\prod_{\bar{g}\in G_{odd}}\prod_{j=1}^{\dim A_{\bar{g}}}Alt_{X_{\phi}(a_{\bar{g}}^{(j)})}\circ\prod_{\bar{g}\in G_{even}}\prod_{j=1}^{\dim A_{\bar{g}}}Sym_{X_{\phi}(a_{\bar{g}}^{(j)})}\circ\prod_{\bar{g}\in G_{odd}}\prod_{i=1}^{kT}Sym_{X_{\bar{g},i}}\circ\prod_{\bar{g}\in G_{even}}\prod_{i=1}^{kT}Alt_{X_{\bar{g},i}}(f_{k}).$$

We also consider a ``forgetful'' operator $F_{G}^{G_{2}}$ which
transforms $G_{2}$-graded polynomials into $G$-graded polynomials
by changing the degree of every variable from $(\epsilon,g)\in G_{2}$
to $g\in G$. We finally have the $G$-graded polynomial
\[
F_{G}^{G_{2}}\left(s_{k;\phi;A}(f)\right).
\]
We remark that for $g\in G$ the variables $F_{G}^{G_{2}}(x_{(0,g),t})$
and $F_{G}^{G_{2}}(x_{(1,g),t})$ are two different variables degree
$g \in G$.
\begin{definition}
Let $B$ be a $G_{2}$-graded algebra. An evaluation of a $G$-graded
polynomial $f$ on $B$ is called \emph{almost $G_{2}$ }if every
variable $x$ of $f$ of degree $g$ is evaluated in some $B_{(\epsilon,g)}$.

Furthermore, if $B_{0}$ is a subset of $B$, we say that an evaluation
$\psi$ of $f$ on $B$ is a \emph{$B_{0}$-evaluation} if every variable
of $f$ is evaluated in $B_{0}$.
\end{definition}

Suppose we have a $G_{2}$-graded polynomial $f$ and we consider
the $G$-polynomial $F_{G}^{G_{2}}(f)$. Then if $\psi$ is an almost
$G_{2}$-evaluation of $F_{G}^{G_{2}}(f)$, there is typically no
reason that $\deg\psi(F_{G}^{G_{2}}(x_{\bar{g}}))=\bar{g}$ (i.e.
the parities might not agree). The next Proposition shows that our
construction of $F_{G}^{G_{2}}\left(s_{k;\phi;A}\right)(f)$
will ensure that ``almost always'' the above equality occurs \textit{given
that} $\psi$ gives a nonzero value to the polynomial.

\begin{proposition}\label{proposition:forcing G evaluation}
Let $B$ be a finite dimensional
$G_{2}$-graded algebra.
If  $$\psi:F\left\langle X_{0};Y_{0}\right\rangle \to E(B)$$
is a nonzero almost $G_{2}$-evaluation, then for every $\bar{g}\in G_{2}$
we have
\[
\deg\psi(F_{G}^{G_{2}}(x_{\bar{g},i}))=\bar{g},
\]
except possibly $(\dim A)^{2}\dim B$ number of $i$'s.

Furthermore, if there is some $\bar{g}_{0}=(\epsilon_{0},g_{0})\in G_{2}$
such that the dimension of $B_{g_{0}}$ is strictly smaller than that
of $A_{g_{0}}$. Then for $k>(\dim A)^{2}\dim B$, the polynomial
$F_{G}^{G_{2}}\left(s_{k;\phi;A}(f)\right)$ is an identity of $E(B)$.
\end{proposition}

\begin{proof}
We focus on proving the ``furthermore'' part and along the way we
get a proof for the main claim. In order to show that $F_{G}^{G_{2}}\left(s_{k;\phi;A}(f)\right)$
is an identity of $E(B)$, it is enough to show that it is $0$ under
any almost $G_{2}$-evaluation of $F_{G}^{G_{2}}\left(s_{k;\phi;A}(f)\right)$,
since this polynomial is multilinear. Let $\psi:F\left\langle X_{0};Y_{0}\right\rangle \to E(B)$
be an almost $G_{2}$-evaluation of $F_{G}^{G_{2}}\left(s_{k;\phi;A}(f)\right)$.

Suppose that $\psi\left(F_{G}^{G_{2}}\left(s_{k;\phi;A}(f)\right)\right)\neq0$.
Then, there is some
$$\sigma\in\prod_{\bar{g}\in G_{odd}}\prod_{i=1}^{kT}S_{X_{\bar{g},i}}\cdot\prod_{\bar{g}\in G_{even}}\prod_{i=1}^{kT}S_{X_{\bar{g},i}}$$
such that, under $\psi$, the polynomial
\[
F_{G}^{G_{2}}\left(\prod_{\bar{g}\in G_{odd}}\prod_{j=1}^{\dim A_{\bar{g}}}Alt_{X_{\phi}(a_{\bar{g}}^{(j)})}\circ\prod_{\bar{g}\in G_{even}}\prod_{j=1}^{\dim A_{\bar{g}}}Sym_{X_{\phi}(a_{\bar{g}}^{(j)})}(f_{k}(\sigma(X));Y)\right)\neq0.
\]
Notice that for all $i$, the set $X_{\bar{g},i}$ stays the same
even after applying $\sigma$.

We claim that all \textit{small sets} $F_{G}^{G_{2}}\left(X_{\bar{g}_{0},i}\right)$,
except possibly $(\dim A)^{2}\dim B$ of them, have all of their variables
assigned to elements of degree $\bar{g}_{0}$. Indeed, we only need
to show that the parity is $\epsilon_{0}$. Assume that $\epsilon_{0}=0$
(the proof for $\epsilon_{0}=1$ is similar). If, on the contrary,
there are more than $(\dim A)^{2}$ small sets $F_{G}^{G_{2}}(X_{\bar{g}_{0},i})$
having, at least, one variable who has an odd evaluation; then, as
$k>(\dim A)^{2}\dim B\geq\dim A_{\bar{g}_{0}}\dim A_{odd}\dim B_{(1,g_{0})}$
and Remark \ref{remark:table structure} holds, there is some $l_{0}\in\{1,\dots,\dim A_{\bar{g}_{0}}\}$
such that at least $\dim B_{(1,g_{0})}$ distinct variables from $F_{G}^{G_{2}}\left(X_{\phi}(a_{\bar{g}_{0}}^{(l)})\right)$
are assigned by $\psi$ values from $B_{(1,g_{0})}\otimes E_{1}$.
However as we symmetrize that set, we must get $0$ - a contradiction.
Notice that we have also proved here the main claim.

Denote by $F_{G}^{G_{2}}(X_{\bar{g}_{0},i_{0}})$ a small set with
the property from the previous paragraph. Since $\dim B_{\bar{g}_{0}}<\dim A_{\bar{g}_{0}}$,
the alternation (symmetrization) of size $\dim A_{\bar{g}_{0}}$ must
nullify the polynomial.
\end{proof}
We are now ready to prove Theorem \ref{theorem:different G2 simples different envelopes}:
\begin{proof}[Proof of Theorem \ref{theorem:different G2 simples different envelopes}.]
By \cite{AljHaile}, $A$ and $B$ are $G_{2}$-isomorphic iff
$A$ and $B$ share the same $G_{2}$-identities (see also \cite{BahturinYasumura} for a far reaching generalization of the statement in \cite{AljHaile}). As a result, it
is enough to show that if $A$ and $B$ are not $G_{2}$-PI-equivalent,
then $E(A)$ and $E(B)$ are not $G$-PI-equivalent.

Assume WLOG that there is a multilinear $G_{2}$-polynomial $p(x_{\bar{g}_{1},1},\dots,x_{\bar{g}_{n},n})$
which is an identity of $B$ but not of $A$. We consider the $G_{2}$-graded
basis $\mathcal{B}_{A}=\{a_{\bar{g}}^{(j)}:\bar{g}\in G_{2},j=1,\dots,\dim_{F}A_{\bar{g}}\}$
of $A$ as in \cite{AljHaile} Theorem $1.1$. Let $\phi$ be a nonzero $\mathcal{B}_{A}$-evaluation
of $p$. We may also assume that $\phi(p)=\delta$, where $\delta$
is a nonzero idempotent of $A$. In the next few paragraphs
we are going to construct a $G_{2}$-graded polynomial $f$ from $p$
on which we will perform the construction from the beginning of the
section to obtain a polynomial $F_{G}^{G_{2}}\left(s_{k;\phi;A}(f)\right)$
which will be an identity of $E(B)$ and a nonidentity of $E(A)$.

For $i=1,\dots,n$ let $X(i)=\{x_{\bar{g},i}^{(j)}:\bar{g}\in G_{2},j=1,\dots,\dim A_{\bar{g}}\}$
be disjoint variables from the ones of $p$ and set $\phi(x_{\bar{g},i}^{(j)})=a_{\bar{g}}^{(j)}$.
For every $i$ let $j(i)$ be such that $\phi(x_{\bar{g}_{i},i})=a_{\bar{g_{i}}}^{(j(i))}$.
We identify $x_{\bar{g}_{i},i}$ with $x_{\bar{g}_{i},i}^{(j(i))}$
for every $i$. We set $X_{0}=\coprod_{i=1}^{n}X(i)$.

Similarly to the construction in the proof of Theorem \ref{isomorphic graded semisimple algebras}, one can construct a multilinear
$G_{2}$-monomial $M=M(X_{0};Y)$ with the property that there is
an evaluation $\phi_{Y}$ of the $Y$-variables such that the only
extension of $\phi_{Y}$ to a nonzero $\mathcal{B}_{A}$-evaluation
$\phi_{Z}$ of $M(X_{0};Y)$ must satisfy $\phi_{M}|_{X_{0}}=\phi|_{X_{0}}$(i.e.
$\phi_{M}$ also extends $\phi$) and if $\phi_{M}$ satisfies $\phi_{M}|_{X_{0}}=\phi|_{X_{0}}$
then $\phi_{M}(M)=\delta$. We continue to denote the unique nonzero
evaluation $\phi_{M}$ of $M$ by $\phi$. Furthermore, one can also
arrange that $\phi(M)=\delta$.

Clearly, $\phi(M\cdot p)=\delta$. However, $M\cdot p$ is not multilinear,
and so we make small changes to solve this issue. Consider a new set
of variables $z_{\bar{g}_{1},1},\dots,z_{\bar{g}_{n},n}$ and replace
in $Z$ (only) the variables $z_{\bar{g}_{i},i}$ by $x_{\bar{g}_{i},i}$
for every $i$ and call the new polynomial $M^{\prime}$. Clearly
$M^{\prime}\cdot p$ is multilinear. We extend $\phi$ to include
all the $z$-variables by declaring $\phi(z_{\bar{g}_{i},i})=\phi(x_{\bar{g}_{i},i})$
so that $\phi(Mp)=\phi(M^{\prime}p)=\delta$.

Finally let
\[
f=M^{\prime}\cdot p^{*},
\]
where $*$ is the Grassmann star operation.

We claim that $f$ satisfies properties (1) - (2): By construction
property (2) holds. Hence we are left with verifying property (1).
Indeed, any nontrivial permutation of any of the variables in some
$X(i)$ induces a new evaluation of $f$, which we call $\phi^{\prime}$,
that differs from $\phi$ only on the set $X(i)$. By the construction
of $M$ (and $M^{\prime}$) we get that $\phi^{\prime}(M^{\prime})=0$;
showing property (1).

We now consider our final polynomial $F_{G}^{G_{2}}\left(s_{k;\phi;A}(f)\right)$,
where $k=n\cdot\dim B\cdot(\dim A)^{2}+1$. Notice that it is a $G$-polynomial
and that the construction also extends $\phi$ to an evaluation of
all of $s_{k;\phi;A}(f)$ (a $G_{2}$-graded evaluation!). We claim
that it is an identity of $E(B)$ but not of $E(A)$. It is not an
identity of $E(A)$ since we can consider the following $G$-evaluation
$\psi$ in $E(A)$: for every variable $v$ appearing in $s_{k;\phi;A}(f)$
we set
\[
\psi(F_{G}^{G_{2}}(v))=\phi(v)\otimes w_{v},
\]
where $w_{v}\in E_{\deg v}$ and all $w_{v}$'s are chosen so that
the product of all of them is nonzero. By the definition of $*$,
we have that $\psi\left(F_{G}^{G_{2}}(p^{*})\right)=\delta\otimes\prod_{v\in p}w_{v}$
and so $\psi\left(F_{G}^{G_{2}}(f_{k})\right)=\delta\otimes\prod_{v\in f_{k}}w_{v}$.
By property (1) of $f$ we conclude that
\[
\psi\left(F_{G}^{G_{2}}\left(\prod_{\bar{g}\in G_{odd}}\prod_{i=1}^{kn}Sym_{X_{g,i}}\circ\prod_{\bar{g}\in G_{even}}\prod_{i=1}^{kn}Alt_{X_{\bar{g},i}}(f_{k})\right)\right)=\psi\left(F_{G}^{G_{2}}(f_{k})\right)=\delta\otimes\prod_{v\in f_{k}}w_{v}.
\]
Finally, since $\phi$ gives the same value $a_{\bar{g}}^{(i)}$ for
every variable in $X_{\phi}(a_{\bar{g}}^{(i)})$, we have that
\[
\psi\left(F_{G}^{G_{2}}\left(s_{k;\phi;A}(f)\right)\right)=C\cdot\delta\otimes\prod_{v\in f_{k}}w_{v}\neq0,
\]
where $C=\prod_{\bar{g}\in G_{2}}\prod_{i=1}^{\dim A_{\bar{g}}}\left|X_{\phi}(a_{\bar{g}}^{(i)})\right|!=((kn)!)^{\dim A}$.

We are left with showing that $F_{G}^{G_{2}}\left(s_{k;\phi;A}(f)\right)$
is an identity of $E(B)$. Suppose that $F_{G}^{G_{2}}\left(s_{k;\phi;A}(f)\right)$
is a nonidentity of $E(B)$. Hence there is a nonzero almost $G_{2}$-evaluation $\psi$ on $E(B)$. As $\psi\left(F_{G}^{G_{2}}\left(s_{k;\phi;A}(f)\right)\right)\neq0$,
there are two permutations
\[
\sigma\in\prod_{\bar{g}\in G_{2}}\prod_{i=1}^{\dim A_{\bar{g}}}S_{X_{\phi}(a_{\bar{g}}^{(i)})},\tau\in\prod_{\bar{g}\in G_{2}}\prod_{i=1}^{kn}S_{X_{\bar{g},i}}
\]
such that
\[
\psi\left(F_{G}^{G_{2}}\left(f_{k}(\sigma\tau(X),Y,Z)\right)\right)\neq0.
\]
 Clearly, $\sigma\tau(X_{i})=\sigma(X_{i})$ and $\sigma$ preserves
the $G_{2}$-degree. By Proposition \ref{proposition:forcing G evaluation}
and the choice of $k$, there is some $i_{0}\in\{1,\dots,k\}$ such
that for every $x_{\bar{g}}\in\sigma(X_{i_{0}})$ we have that $\deg\psi(x_{\bar{g}})=\bar{g}$.
As a result, as $p$ is an identity of $B$, we can deduce that $\psi\left(p^{*}(\sigma(X_{i_{0}}))\right)=0$
and so also $\psi\left(f(\sigma(X_{i_{0}}),Y,Z)\right)=0$. This clearly
forces that $\psi\left(F_{G}^{G_{2}}\left(f_{k}(\sigma\tau(X),Y,Z)\right)\right)=0$,
hence reaching a contradiction.
\end{proof}

We may extend Theorem to full $G_{2} = \mathbb{Z}_{2} \times G$-graded algebras.

\begin{theorem}
Let $A$ and $B$ be finite dimensional $G_{2}$-graded algebras over $F$. Suppose $A$ and $B$ are full. If  $E(A)$ and $E(B)$ are $G$-graded PI-equivalent then the semisimple parts $A_{ss}$ and $B_{ss}$ are isomorphic as $G_{2}$-algebras.
\end{theorem}
\begin{proof}
For the proof we shall combine the constructions in Section \ref{nonaffine algebras section} and Section \ref{Section G-graded algebras}, that is for nonaffine ungraded algebras and for affine $G$-graded algebras, together with the Theorem \ref{theorem:different G2 simples different envelopes}. For each $G_{2}$-graded simple algebra $A_{i}$ we let $\mathcal{B}_{A_i}$ be a basis of $A_{i}$ whose elements are $G_{2}$-homogeneous of the form $\{u_{h} \otimes e_{r,s}\}$. Let $K_{i}$ denote a nonzero product of the elements in $\mathcal{B}_{A_i}$.
We refer to these elements as \textit{designated elements}. Each basis element is bordered by basis elements where for convenience we may assume all but possibly one are of the form $1\otimes e_{i,j}$. As usual we refer to these as \textit{frame elements}. We may use one of the frame elements so the value of the monomial is an idempotent $\delta$ of $A$. We denote this product by $Z_{i}$. We let $Z_{i,j}$, $j =1,\ldots, k$ be a duplicate of the monomial $Z_{i}$ and let $\overline{Z}_{i, k} = Z_{i, 1}\cdot Z_{i,2} \cdots Z_{i,k}$. Here, $k$ is a large integer which needs to be determined. We let $\Theta_l = (a,\ldots,a)$ be the $k$-tuple where $a$ is the $l$th element appearing in the monomial $K_{i}$.
 Since the algebra $A$ is full, we have up to ordering of the $G_{2}$-graded simple components of $A$ a nonvanishing product $\overline{Z}_{1, k} \cdot w_{1} \cdot \overline{Z}_{2, k} \cdots w_{q-1} \cdot \overline{Z}_{q, k} \neq 0$.
For every $\bar{g} \in G_{2}$ we consider $k$ \textit{small sets }, each consisting of $dim_{F}(A_{ss})_{\bar{g}}$ designated elements where the $j$th small set consists of the designated elements in $K_{1, j}, \ldots, K_{q, j}$. We have as in previous cases that any nontrivial permutation on a small set leads to a zero product. Our next step is to tensor even elements with the identity of $E$, the Grassmann algebra, and odd elements with different generators of $E$. Note that the product remains nonzero. As in previous cases we will view the elements obtained as $G$-graded elements but for convenience we will still refer to them using the adjective even or odd. Moreover we shall refer as small sets, a set of the form $(1\otimes a_1,\ldots, 1\otimes a_{m})$ where  $(a_1,\ldots, a_{m})$ is a small set of even homogeneous elements of degree $(0,g)$, $g \in G$ or a set the form $(\epsilon_{1}\otimes b_1,\ldots, \epsilon_{m}\otimes b_{m})$ where  $(b_1,\ldots, b_{m})$  is a small set of odd homogeneous elements of degree $(1,g)$, $g \in G$. By abuse of notation we keep the notation $\Theta_l$ after multiplying the basis elements with Grassmann generators.

Next we alternate and symmetrize small sets of even and odd elements respectively. Then we symmetrize sets $\Theta_l = (a,\ldots,a)$ where $a$ is even and alternate sets $\Theta_t = (b, \ldots, b)$ where $b$ is odd. One shows the product is nonzero.

Next we replace the the designated elements by $X$ variables, the frames by $Y$'s and the bridges by $W$'s where we forget the $\mathbb{Z}_{2}$-degree, that is $X, Y, W$ are $G$-graded variables. Clearly by construction we have a nonidentity $f$ of $A$. Let us denote the nonzero evaluation above by $\phi$. As in previous cases with such polynomial one shows that if $B_{ss}$  does not cover $A_{ss}$ as $G_{2}$-algebras then $f$ is a nonidentity of $E(A)$ and an identity of $E(B)$ as a $G$-graded algebras. Thus, since we are assuming $E(A)$ and $E(B)$ are $G$-graded PI-equivalent we have that $A$ and $B$ cover each other as $G_{2}$-graded algebras. We conclude that up to permutation of the simple components of $B$ we have $A \cong A_{1} \times \ldots \times A_{q} \oplus J_{A}$ and $B \cong B_{1} \times \ldots \times B_{q} \oplus J_{B}$ where $dim_{F}(A_{j})_{g} = dim_{F}(B_{j})_{g}$, $g \in G_{2}$.
We want to prove there is a permutation on the $G_{2}$-graded simple components of $B$ such that $A_{j} \cong B_{j}$ as $G_{2}$-graded algebras.

Recall from the Theorem \ref{theorem:different G2 simples different envelopes} above that if $dim_{F}(A_{j})_{g} = dim_{F}(B_{j'})_{g}$, all $g \in G_{2}$, for some $j$ and $j'$, there exists a $G$-polynomial $p_{j, j'}$ which is a $G$-graded nonidentity of $E(A_{j})$ and an identity of $E(B_{j'})$ unless $A_{j}$ and $B_{j'}$ are $G_{2}$-graded isomorphic. Moreover, we may assume the value of the polynomial $p_{j,j'}$ is the idempotent $\delta$ of $A$ we fixed above. Denote by $p_{i} = \prod_{j'} p_{i,j'}$. We note that $p_{i}$ is a $G$-polynomial nonidentity of $E(A_{i})$ and an identity of $E(B_{j'})$ for every $G_{2}$-graded simple algebra whose dimension of the homogeneous $G_{2}$-components are equal to the corresponding dimensions of the homogeneous components of $A_{j}$ but is not isomorphic to $A_{j}$. Finally, we insert to the right of every monomial $Z_{i,l}$ a copy of the polynomial $p_{i}$ with disjoint variables. The polynomial obtained $m_{A}$ is a $G$-graded nonidentity of $E(A)$. By assumption it is a nonidentity of $E(B)$, which forces the existence of a permutation on the $G_{2}$-graded simple components of $B$ such that $A_{j} \cong B_{j}$ as $G_{2}$-graded algebras. This completes the proof.
\end{proof}

We can now complete the proof of Theorem \ref{uniquely determined graded nonaffine} as in the proof of Theorem \ref{uniquely determined affine}, that is by performing Steps $0-4$ on the set of finite dimensional $G_{2}$-graded algebras $A$ with $Id_{G}(E(A)) = \Gamma$ (see Section \ref{Section preliminaries and proof of the affine case}). Details are omitted.

\end{document}